\documentclass[12pt]{amsart}
\usepackage[english]{babel}
\usepackage{graphicx}
\usepackage{amsmath,amssymb,hyperref,srcltx}
\newcommand{\diff}[2]{\mbox{{\rm Diff}{${\,}_{#1}({\mathbb C}^{#2},0)$}}}
\newcommand{\diffh}[2]{\mbox{$\widehat{\rm Diff}{{\,}_{#1}({\mathbb C}^{#2},0)}$}}

\newcommand{\cn}[1]{\mbox{(${\mathbb C}^{#1},0$)}}

\newtheorem{pro}{Proposition}[section]
\newtheorem{teo}{Theorem}[section]

\newtheorem{cor}{Corollary}[section]

\newtheorem{lem}{Lemma}[section]

\theoremstyle{definition}
\newtheorem{defi}{Definition}[section]

\theoremstyle{remark}
\newtheorem{rem}{Remark}[section]

\begin{document}

\title[]
{Derived length of solvable groups of local diffeomorphisms}

\author{Mitchael Martelo and Javier Rib\'{o}n}
\address{Instituto de Matem\'{a}tica, UFF, Rua M\'{a}rio Santos Braga S/N
Valonguinho, Niter\'{o}i, Rio de Janeiro, Brasil 24020-140}
\thanks{e-mail address: mitchaelmartelo@id.uff.br, javier@mat.uff.br}
\thanks{MSC-class. Primary: 37F75, 20F16; Secondary: 20F14, 32H50}
\maketitle

\bibliographystyle{plain}
\section*{Abstract}
Let $G$ be a solvable subgroup of the group $\diff{}{n}$
of local complex analytic diffeomorphisms.
Analogously as for groups of matrices we bound the solvable
length of $G$ by a function of $n$.
Moreover we provide the best possible bounds for connected, unipotent and
nilpotent groups.
\section{Introduction}
Let $K$ be a field. The soluble length (or equivalently derived length) of a
solvable matrix subgroup of $GL(n,K)$ is bounded by a function of the dimension $n$.
The existence of such bound was observed by Malcev \cite{Malcev}
and Zassenhaus.
The value for the best possible upper bound was provided by
Newman \cite{Newman}. The examples of groups of maximal length can
be obtained for instance for $K=  {\mathbb C}$.

Let ${\rm Diff} ({\mathbb C}^{n},0)$ be the group of germs of complex analytic
diffeomorphisms defined in
a neighborhood of $0 \in {\mathbb C}^{n}$.
We are interested on studying these classical topics in the context of
subgroups of
$\diff{}{n}$ and more generally for subgroups of the formal completion
$\diffh{}{n}$ of  $\diff{}{n}$.

Groups of diffeomorphisms behave  differently than matrix groups with respect to
algebraic properties.
It is known that for an algebraically closed field $K$ of characteristic $0$
the nilpotent class of a nilpotent subgroup of $GL(n,K)$ is not bounded
by a function of $n$ (prop. 8.3, page 103 \cite{Wehrfritz}).
Anyway if $G$ is a connected or a unipotent subgroup of $GL(n,K)$
then such nilpotent class is bounded by $n-1$ ($n >1$) as a consequence of
Lie-Kolchin's theorem. Consider the unipotent subgroup
\[ G = \left\{ \left(
\frac{x + P_{k}(y)}{(1+ty)^{k}} , \frac{y}{1+ty}
\right) : P_{k} \in {\mathbb C}[y] \cap (y^{2}), \ \deg P_{k} \leq k  \right\} \]
of $\diff{}{2}$. The nilpotent class of $G$ is $k-1$.
It is easy to see that the $jth$ group
${\mathcal C}^{j} G$ of the descending central series of $G$
(see def. \ref{def:com}) for $1 \leq j \leq k-2$ is non trivial and
 it is composed of elements of the form
 $\phi(x,y)= (x +Q_{j}(y,z), y)$ where
 $Q_{j} \in  {\mathbb C}[y]$ is a polynomial in $y$
 of the form $a_{j+2} y^{j+2} + \hdots + a_{k} y^{k}$.
 Thus there is no hope of bounding the nilpotent class of unipotent groups
 in terms of the  dimension. In spite of this there exists a bound
 for the derived length of solvable groups.
\begin{teo}
\label{teo:mainc2}
Let $G \subset \diffh{}{n}$ be a solvable group.
Then the soluble length of  $G$ is at most $2n -1 + \rho(n)$ where
$\rho : {\mathbb N} \to {\mathbb N}$ is the Newman function \cite{Newman}.
\end{teo}
The solubility of groups of diffeomorphism is a topic related to the integrability
of complex analytic foliations.
Suppose that ${\mathcal L}$ is a leaf of a complex analytic foliation
${\mathcal F}$ of codimension $q$. We choose a point $p \in {\mathcal L}$ and
a local transversal $T$ to the foliation through the point $p$.
We can associate to any $\gamma \in \pi_{1}({\mathcal L},p)$ a holonomy mapping
defined in the neighborhood of $p$ in $T$. By identifying a neighborhood of $p$
in $T$ with a neighborhood of the origin in ${\mathbb C}^{p}$ we obtain a homomorphism
of groups
\[  \pi_{1}({\mathcal L},p) \to  \diff{}{q}  \]
whose image is the holonomy group of the leaf ${\mathcal L}$.
Roughly speaking in the codimension $1$ case the existence of a Liouvillian
first integral is equivalent to the solubility of holonomy groups.
The first result in this direction is probably the Mattei-Moussu's
topological characterization for the existence of holomorphic first integrals
of germs of complex analytic codimension $1$ foliations  \cite{MaMo:Aen}
(namely the leaves
are closed and there exists a finite number of leaves adhering to the origin)
whose proof is based on taking advantage of the finitude of the
holonomy groups that appear naturally in the desingularization process.
To the interested reader
we recommend E. Paul's paper  \cite{Paul:pre} where
the existence and
nature of Liouvillian first integrals is characterized in terms of the properties of
the projective holonomy groups. In this context it is interesting to study the properties
of solvable groups of local diffeomorphisms as an attempt to understand the
integrability properties of higher codimension foliations.

The analysis of solvable groups of local diffeomorphisms is a valuable tool to
describe solvable groups of real analytic difffeomorphisms of compact manifolds.
Ghys uses this strategy to prove that solvable groups of real analytic diffeomorphisms
of ${\mathbb S}^{1}$ and nilpotent groups of real analytic diffeomorphisms of ${\mathbb S}^{2}$
are metabelian, i.e. their first derived groups are commutative \cite{Ghys-identite}.
In the same spirit Cantat and Cerveau study
analytic actions of mapping class groups on surfaces \cite{Cantat-Cerveau}.

It is well known that
the bound in theorem \ref{teo:mainc2} is optimal for $n=1$ since the group
\[ G = \left\{ \frac{\lambda z}{1 + t z} : \lambda \in {\mathbb C}^{*}, \ t \in {\mathbb C}
\right\} \subset \diff{}{} \]
is of soluble length $2$. Indeed the transformation $\psi : G \to Aff ({\mathbb C})$
onto the affine group defined as
$\lambda z / (1 + t z)  \to \lambda z + t$ is an automorphism.
The bound could be non-optimal for $n >1$.
In spite of this we provide optimal upper bounds for the soluble length in
several classes of groups,  namely connected, unipotent and nilpotent groups.

We define $l(G)$ (resp. $l({\mathfrak g})$) the soluble length of a group
(resp. a Lie algebra), see def. \ref{def:dlie} and \ref{def:length}.
\begin{teo}
\label{teo:mainc}
Let $G \subset \diffh{}{n}$ be a solvable group such that $j^{1} G$ is connected.
Then we have $l(G) \leq 2n$. Moreover the length $2n$
is attained for a subgroup $G$ of $\diff{}{n}$.
\end{teo}
We associate a Lie algebra to $G$ whose solvability properties are the same.
Of course this approach can be successful only if the group $G$ is somehow
connected. We see that it suffices to require that the subgroup
$j^{1} G$ of $GL(n,{\mathbb C})$ be connected in the usual or the Zariski topologies.
Anyway the proper definition of connectedness allows to generalize the theorem
to solvable groups such that the closure in the Zariski topology of
$j^{1} G$ is connected.

The theorem \ref{teo:mainc} is a consequence of its mirror for Lie algebras.
We denote by
${\mathcal X} \cn{n}$ the set of germs of complex analytic
vector fields which are singular at $0$.
The formal completion of this space is denoted by
$\hat{\mathcal X} \cn{n}$.
\begin{teo}
\label{teo:mainca}
Let ${\mathfrak g} \subset \hat{\mathcal X} \cn{n}$ be a solvable Lie algebra.
Then we have $l({\mathfrak g}) \leq 2n$. Moreover the length $2n$
is attained for a Lie subalgebra ${\mathfrak g}$ of ${\mathcal X} \cn{n}$.
\end{teo}
Next we are interested on improving the estimates in theorem \ref{teo:mainc}
for particular classes of groups.
We denote $\diff{u}{n}$ the set of unipotent elements of $\diff{}{n}$, more precisely
$\varphi \in \diff{u}{n}$ if $j^{1} \varphi$ is a unipotent linear isomorphism
(i.e. $j^{1} \varphi - Id$ is nilpotent).
 We denote $\diffh{u}{n}$
the formal completion of $\diff{u}{n}$.
We say that a group $G \subset \diff{}{n}$ is unipotent if
it is composed of unipotent elements.
A unipotent group $G$ can  have a linear part $j^{1} G$ that is not connected
but we can replace $G$ with a sort of algebraic closure $\overline{G}^{(0)}$
such that $G \subset \overline{G}^{(0)}$, $l(G)=l(\overline{G}^{(0)})$ and
$j^ {1} \overline{G}^{(0)}$ is connected. Essentially unipotent groups are
always connected.
\begin{teo}
\label{teo:main}
Let $G \subset \diffh{}{n}$ be a unipotent solvable group.
Then we have $l(G) \leq 2n-1$. Moreover there exists a unipotent solvable group
$G \subset \diff{}{n}$ such that $l(G)=2n-1$.
\end{teo}
Theorem \ref{teo:mainc2} is a corollary of
theorem \ref{teo:main}.  Let $G \subset \diffh{}{n}$ be a solvable group. Since $j^{1} G$ is a solvable
linear group then $G^{(\rho(n))}$ (see def. \ref{def:com})
is composed of transformations whose linear part
is the identity. Thus the group $G^{(\rho(n))}$ is unipotent.
Since the soluble length of $G^{(\rho(n))}$ is at most $2n-1$ by
theorem  \ref{teo:main} then the soluble length of $G$ is less or
equal than $2n -1 + \rho(n)$.
%

Theorem \ref{teo:mainc} can be much improved
if the group $G$ is nilpotent.
\begin{teo}
\label{teo:mainn}
Let $G \subset \diffh{}{n}$ be a nilpotent group.
Then we have $l(G) \leq n$. Moreover there exists a unipotent nilpotent subgroup
$G$ of $\diff{}{n}$ such that $l(G)=n$.
\end{teo}
Let us notice that $G$ is not required to be connected in the previous theorem.
The bound $l(G) \leq 2$ was proved by Ghys in \cite{Ghys-identite} for the case $n=2$.
The first author obtained the property  $l(G) \leq n$
for nilpotent groups of tangent to the identity diffeomorphisms \cite{Martelo}.
Notice that unlike for solvable groups the estimate is not different for
non-connected, connected and unipotent groups.

Let us say a word about the proof of the theorems. The main difficulty of working with
Lie algebras of vector fields (and groups of diffeomorphisms) is that they are not
finite dimensional.  In section \ref{sec:lie} we construct the algebraic closure
$\overline{G}^{(0)}$ of a group $G \subset \diffh{}{}$.
The group $\overline{G}^{(0)}$
is a projective limit of finite dimensional linear algebraic groups.
It has analogous algebraic properties as $G$, in particular their soluble lengths coincide.
As in the finite dimensional theory
we can associate a Lie algebra ${\mathfrak g}$ to its connected component of the identity
$\overline{G}_{0}^{(0)}$ (see def. \ref{def:liegrp}).
We prove that the soluble length of $\overline{G}_{0}^{(0)}$ and ${\mathfrak g}$
are the same. Since we have $\overline{G}_{0}^{(0)}=\overline{G}^{(0)}$
if $G$ is unipotent or $j^{1} G$ is connected we can work with Lie algebras of vector
fields in order to prove theorems \ref{teo:mainc}, \ref{teo:main} and
\ref{teo:mainn}. Section \ref{sec:sollen} is devoted to bound the derived length of
solvable Lie algebras of vector fields.
We interpret the vector fields as actions on some finite dimensional spaces of
meromorphic first integrals of some Lie subalgebras in the derived series.
By using these ideas we associate Lie algebras of matrices
to any solvable Lie algebra.
We can apply the Lie theorem (th. \ref{teo:Lie}) to the Lie algebras of matrices
in order to obtain the patterns that appear in the reduction of the derived series.
In section \ref{sec:exa} we prove that the bounds for the soluble length
are optimal by providing examples.
These examples follow the patterns presented in section \ref{sec:sollen}.
\section{Notations}
Let ${\rm Diff} ({\mathbb C}^{n},0)$ be the group of germs of complex
analytic diffeomorphisms defined in a neighborhood of $0 \in {\mathbb C}^{n}$.
Consider the maximal ideals ${\mathfrak m}_{0}$ and ${\mathfrak m}$
of  ${\mathbb C}\{x_{1},\hdots,x_{n}\}$ and  ${\mathbb C}[[x_{1},\hdots,x_{n}]]$
respectively. Let $\diffh{}{n}$ be
the formal completion of $\diff{}{n}$
with respect to the filtration
$\{{\mathfrak m}_{0}^{k} \times \hdots \times {\mathfrak m}_{0}^{k} \}_{k \in {\mathbb N} \cup \{0\}}$
(see \cite{Eisenbud}, section 7.1). The composition in
$\diffh{}{n}$ is defined in the natural way by taking the composition in
$\diff{}{n}$ and passing to the limit in the Krull topology
(see \cite{Eisenbud}, page 204).
We denote by $\diff{u}{n}$ the subgroup of unipotent elements of $\diff{}{n}$, more precisely
$\varphi \in \diff{u}{n}$ if $j^{1} \varphi$ is a unipotent linear isomorphism
(i.e. $j^{1} \varphi - Id$ is nilpotent).
 Analogously we denote $\diffh{u}{n}$
the formal completion of $\diff{u}{n}$.

We denote by
${\mathcal X} \cn{n}$ the set of germs of complex analytic
vector fields which are singular at $0$. We denote
by ${\mathcal X}_{N} \cn{n}$ the subset of ${\mathcal X} \cn{n}$ of nilpotent vector fields, i.e.
vector fields whose first jet has the unique eigenvalue $0$.
The formal completions of these spaces are denoted by
$\hat{\mathcal X} \cn{n}$ and $\hat{\mathcal X}_{N} \cn{n}$ respectively.

The expression
\begin{equation}
\label{equ:exp}
{\rm exp} (t \hat{X}) = \left({
\sum_{j=0}^{\infty} \frac{t^{j}}{j!} \hat{X}^{j}(x_{1}), \hdots,
\sum_{j=0}^{\infty} \frac{t^{j}}{j!} \hat{X}^{j}(x_{n}) }\right)
\end{equation}
defines the exponential of $t \hat{X}$ for $\hat{X} \in \hat{\mathcal X} \cn{n}$
and $t \in {\mathbb C}$.
Let us remark that $\hat{X}^{j}(g)$ is the result of applying $j$ times the derivation
$\hat{X}$ to the power series $g$. The definition coincides
with the classical one if $\hat{X}$ is a germ of convergent vector field.
Indeed given an element
\[ \hat{X} = \hat{a}_{1} \frac{\partial}{\partial x_{1}} + \hdots + \hat{a}_{n}
\frac{\partial}{\partial x_{n}} \in \hat{\mathcal X} \cn{n} \]
and $X =  {a}_{1} \partial /\partial x_{1} + \hdots + {a}_{n} \partial /\partial x_{n}
\in {\mathcal X} \cn{n}$ such that $\hat{a}_{j} - a_{j} \in {\mathfrak m}^{k+1}$
for any $1 \leq j \leq n$
then the $k$-jet $j^{k} {\rm exp}(t \hat{X})$ of ${\rm exp}(t \hat{X})$
is equal to $j^{k} {\rm exp}(t {X})$.

For
$\hat{X}$ in $\hat{\mathcal X}_{N} \cn{n}$ the sums defining the components of
${\rm exp}(t \hat{X})$
converge in the Krull topology of ${\mathbb C}[[x_{1},\hdots,x_{n}]]$,
i.e. the multiplicity at the origin of
$\hat{X}^{j}(g)$ tends to $\infty$ when $j \to \infty$ for any
$g \in {\mathbb C}[[x_{1},\hdots,x_{n}]]$.
The unipotent germs of diffeomorphisms are related with nilpotent
vector fields;  the next proposition is classical.
\begin{pro}
\label{pro:clas}
(see \cite{Ecalle}, \cite{MaRa:aen})
The exponential mapping ${\rm exp}$ induces a bijection from
$\hat{\mathcal X}_{N} \cn{n}$ onto $\diffh{u}{n}$.
\end{pro}
\begin{defi}
Let $\varphi \in \diffh{u}{n}$. We denote by $\log \varphi$ the unique element of
$\hat{\mathcal X}_{N} \cn{n}$
such that $\varphi = {\rm exp}(\log \varphi)$. We say that
$\log \varphi$ is the infinitesimal generator of $\varphi$.
\end{defi}
In general the infinitesimal generator of a germ of diffeomorphism
is a divergent vector field (see \cite{Ah-Ro}).
\section{Lie algebras of connected and unipotent groups}
\label{sec:lie}
Let $G \subset \diff{}{n}$ be a group of local diffeomorphisms.
We are interested on studying its solvability properties.
The next definitions are included for the sake of clarity.
\begin{defi}
\label{def:com}
Let $G$ be a group. We define
$[\alpha, \beta] = \alpha \beta \alpha^{-1} \beta^{-1}$
the commutator of $\alpha$ and $\beta$.
Consider subgroups $H$ and $L$ of $G$.
We define $[H,L]$ the subgroup of $G$ generated by the elements of the
form $[\alpha,\beta]$ for $\alpha \in H$ and $\beta \in L$. We define
\[ G^{(0)} = G, \  G^{(n+1)} = [G^{(n)},G^{(n)}]  \ \forall n \geq 0\]
the derived series of $G$. We define
\[ {\mathcal C}^{0} G = G, \  {\mathcal C}^{n+1} G = [G, {\mathcal C}^{n} G]  \ \forall n \geq 0\]
the descending central series of $G$.
\end{defi}
\begin{defi}
\label{def:dlie}
Let ${\mathfrak g}$ be a complex Lie algebra.
Denote by $[X,Y]$ the Lie bracket of $X,Y \in {\mathfrak g}$.
Consider Lie subalgebras ${\mathfrak h}$ and ${\mathfrak l}$ of
${\mathfrak g}$.
We define $[{\mathfrak h},{\mathfrak l}]$ the
Lie subalgebra of ${\mathfrak g}$
generated by the elements of the
form $[X,Y]$ for $X \in {\mathfrak h}$ and $Y \in {\mathfrak l}$. We define
\[ {\mathfrak g}^{(0)} = {\mathfrak g}, \
{\mathfrak g}^{(n+1)} = [{\mathfrak g}^{(n)},{\mathfrak g}^{(n)}]  \ \forall n \geq 0\]
the derived series of ${\mathfrak g} $. We define
\[ {\mathcal C}^{0} {\mathfrak g} = {\mathfrak g}, \
{\mathcal C}^{n+1} {\mathfrak g} = [{\mathfrak g}, {\mathcal C}^{n} {\mathfrak g}]  \ \forall n \geq 0\]
the descending central series of ${\mathfrak g}$.
\end{defi}
\begin{defi}
\label{def:length}
Let $G$ be a group (resp. a Lie algebra).
We define $l(G)$ the soluble length of $G$
as
\[ l(G) = \min \{ k \in {\mathbb N} \cup \{0 \} :  G^{(k)} = \{ Id \} \} \]
where $\min \emptyset = \infty$. We say that $G$ is solvable if
$l(G) < \infty$. We say that $G$ is nilpotent if there exists
$j \geq 0$ such that ${\mathcal C}^{j} G = \{ Id \}$.
If $j$ is the minimum non-negative integer number with such a property
we say that $G$ is of nilpotent class $j$.
\end{defi}
In order to analyze the properties of $l(G)$ for $G \subset \diff{}{n}$
we can study the group
$j^{1} G$ of linear parts of elements of $G$. Let
$\rho: {\mathbb N} \to {\mathbb N}$ be the Newman function
(\cite{Newman}, see also the introduction).
Then the $\rho(n)$-derived group $G^{(\rho(n))}$ is a group of local diffeomorphisms
tangent to the identity. It is natural to study the soluble length of
such groups. In this paper we deal with bigger classes of groups.
\begin{defi}
Let $G$ a subgroup of $\diffh{}{n}$. We say that $G$ is {\it connected}
if the closure in the Zariski topology of the subgroup of linear parts
$j^{1} G \subset GL(n,{\mathbb C})$ of $G$ is connected.
\end{defi}
\begin{defi}
Let $G$ a subgroup of $\diffh{}{n}$. We say that $G$ is {\it unipotent}
if $G \subset \diffh{u}{n}$.
\end{defi}
This kind of groups is interesting in itself. For instance given a solvable
subgroup $G$ of $\diffh{}{n}$ we can consider the subset $G_{u}$ of
$G$ composed by the unipotent elements of $G$. The solvable nature
of $G$ implies that $G_{u}$ is a group. Thus $G_{u}$ is a normal subgroup
of $G$. We can consider $G$ as a group of automorphisms of $G_{u}$
acting by conjugation. The structure of the solvable group $G_{u}$
helps to determine the solubility properties of $G$.
%
%
%
%

Ghys associates a Lie algebra of formal nilpotent vector fields
to any group of unipotent diffeomorphisms (prop. 4.3 in \cite{Ghys-identite}).
In the same spirit we present a construction that
associates a Lie subalgebra of $\hat{\mathcal X}  \cn{n}$
to  any  subgroup $G$ of $\diffh{}{n}$. We replace
$G$ with a subgroup
$\overline{G}^{(0)}$ of $\diffh{}{n}$ containing $G$ that is, roughly speaking,
the algebraic closure of $G$. Such a group satisfies
$l(\overline{G}^{(0)})=l(G)$ and it has a natural Lie algebra.
In particular the construction associates a Lie algebra, with analogous
algebraic properties, to every connected group of diffeomorphisms.

 Let ${\mathfrak m}$ the maximal ideal of
${\mathbb C}[[x_{1},\hdots,x_{n}]]$.
Any formal diffeomorphism $\varphi \in \diffh{}{n}$ acts on the
finite dimensional complex vector space ${\mathfrak m}/{\mathfrak m}^{k+1}$ of $k$-jets.
More precisely $\varphi$ defines an element
$\varphi_{k}$ of $GL({\mathfrak m}/{\mathfrak m}^{k+1})$ given by
\[
\begin{array}{ccc}
{\mathfrak m}/{\mathfrak m}^{k+1} &
\stackrel{\varphi_{k}}{\rightarrow} & {\mathfrak m}/{\mathfrak m}^{k+1} \\
g + {\mathfrak m}^{k+1}& \mapsto & g \circ \varphi + {\mathfrak m}^{k+1}
\end{array} .
\]
Analogously a formal vector field $X \in \hat{\mathcal X} \cn{n}$
defines an element $X_{k}$ of $GL({\mathfrak m}/{\mathfrak m}^{k+1})$ given by
\begin{equation}
\label{equ:actvf}
\begin{array}{ccc}
{\mathfrak m}/{\mathfrak m}^{k+1} &
\stackrel{X_{k}}{\rightarrow} & {\mathfrak m}/{\mathfrak m}^{k+1} \\
g + {\mathfrak m}^{k+1}& \mapsto & X(g) + {\mathfrak m}^{k+1}
\end{array} .
\end{equation}
Consider the group $D_{k} \subset GL({\mathfrak m}/{\mathfrak m}^{k+1})$
defined as
\[ D_{k} = \{ \alpha \in GL({\mathfrak m}/{\mathfrak m}^{k+1}) :
\alpha (g h) = \alpha (g) \alpha (h) \ \forall g,h \in  {\mathfrak m}/{\mathfrak m}^{k+1} \} .  \]
The equations of the form $\alpha (g h) = \alpha (g) \alpha (h)$ are algebraic
in the coefficients of $\alpha$. Thus $D_{k}$ is an algebraic group, indeed it is
the subgroup $\{ \varphi_{k} : \varphi \in \diffh{}{n} \}$ of actions on
${\mathfrak m}/{\mathfrak m}^{k+1}$ given by formal diffeomorphisms.
Fix $k \in {\mathbb N}$. An element $\alpha$ of
$D_{k+1}$ satisfies
$\alpha ({\mathfrak m}^{k+1}/ {\mathfrak m}^{k+2}) = {\mathfrak m}^{k+1}/ {\mathfrak m}^{k+2}$.
Therefore $\alpha$ induces a unique element in $D_{k}$. In this way
we  define a mapping $\pi_{k} :D_{k+1}  \to D_{k}$.

We define the group
$C_{k}=\{ \varphi_{k} : \varphi \in G\} \subset GL({\mathfrak m}/{\mathfrak m}^{k+1})$.
We consider the matrix group $G_{k}$
defined as the smallest algebraic subgroup of $GL({\mathfrak m}/{\mathfrak m}^{k+1})$
containing $C_{k}$.
We have $\pi_{k}(C_{k+1})=C_{k}$ for any $k \in {\mathbb N}$.
The mapping $\pi_{k} :D_{k+1}  \to D_{k}$ is a morphism of algebraic groups.
Thus the image by $\pi_{k}$ of the smallest algebraic subgroup of $D_{k+1}$
containing $C_{k+1}$ is the smallest algebraic subgroup of $D_{k}$ containing
$C_{k}= \pi_{k}(C_{k+1})$ (see 2.1 (f), page 57 \cite{Borel}).
In other words we obtain $\pi_{k} (G_{k+1})=G_{k}$ for any $k \in {\mathbb N}$.
\begin{defi}
\label{def:liegrp}
Let $G$ be a subgroup of $\diffh{}{n}$.
We denote $G_{k,0}$ the connected component of the identity of
$G_{k}$ for any $k \in {\mathbb N}$. We define
\[ \overline{G}^{(0)} = \{ \varphi \in \diffh{}{n} : \varphi_{k} \in G_{k} \ \forall k \in {\mathbb N} \} . \]
Clearly $\overline{G}^{(0)}$ is a subgroup of $\diffh{}{n}$ containing $G$ that is closed
in the Krull topology. We define
\[ \overline{G}_{0}^{(0)} = \{ \varphi \in \diffh{}{n} : \varphi_{k} \in G_{k,0} \
\forall k \in {\mathbb N} \} . \]
We can consider $\overline{G}_{0}^{(0)}$ as
the connected component of the identity of $\overline{G}^{(0)}$.
\end{defi}
\begin{defi}
\label{def:liegrp2}
Let $G$ be a subgroup of $\diffh{}{n}$.
We define $G_{k,s}$ as the subset of $G_{k}$ of diagonalizable elements.
We define $G_{k,u}$ as the subset of $G_{k}$ of unipotent elements.
By the Jordan-Chevalley  decomposition for algebraic groups any element $\alpha$
of $G_{k}$ can be expressed in the form
\[ \alpha = \alpha_{s} \circ \alpha_{u} =  \alpha_{u} \circ \alpha_{s} \]
for commuting $\alpha_{s} \in G_{k,s}$ and  $\alpha_{u} \in G_{k,u}$
in a unique way. We define
\[ \overline{G}_{s}^{(0)} = \{ \varphi \in \diffh{}{n} : \varphi_{k} \in G_{k,s}
\ \forall k \in {\mathbb N} \}  \]
and
\[ \overline{G}_{u}^{(0)} = \{ \varphi \in \diffh{}{n} : \varphi_{k} \in G_{k,u}
\ \forall k \in {\mathbb N} \} . \]
\end{defi}
\begin{rem}
The elements of $\overline{G}_{s}^{(0)}$ are formally linearizable.
It is well-known and a simple exercise that
$\varphi \in \diffh{u}{n}$ (or equivalently $\varphi_{1}$ is unipotent)
if and only if $\varphi_{k}$ is unipotent for any $k \in {\mathbb N}$.
Thus we obtain
$\overline{G}_{u}^{(0)} = \overline{G}^{(0)}  \cap  \diffh{u}{n}$.
\end{rem}
The next lemmas are intended to introduce several basic facts about connected
groups. For instance $G_{k}$ is always connected for $k \in {\mathbb N}$.
This makes these groups well represented by their Lie algebras at any finite jet
level. Their algebraic properties can be translated to the Lie algebra setting
where it is easier to work.
\begin{lem}
\label{lem:elpro}
Let $G$ be a subgroup of $\diffh{}{n}$.
Then
\begin{itemize}
\item $l(\overline{G}^{(0)})=l(G)$.
\item $G_{k,u} \subset G_{k,0}$ for any $k \in {\mathbb N}$.
\item ${\rm exp}(t \log \varphi) \in \overline{G}_{0}^{(0)}$ for all
$\varphi \in \overline{G}_{u}^{(0)}$ and $t \in {\mathbb C}$.
\end{itemize}
\end{lem}
\begin{proof}
Since $l(G) \leq l(\overline{G}^{(0)})$ it suffices to prove that
$G^{(p)} = \{Id\}$ implies $(\overline{G}^{(0)})^{(p)} = \{Id\}$.
The property $G^{(p)} = \{Id\}$ is equivalent to a system of algebraic equations
in each space of jets. We obtain $G_{k}^{(p)} = \{Id\}$ for any $k \in {\mathbb N}$ and
then $(\overline{G}^{(0)})^{(p)} = \{Id\}$.

The second result is classical in characteristic $0$
(see lemma C, page 96 \cite{Humphreys}).
Fix a unipotent element $\alpha$ of $G_{k}$. There exists
a nilpotent element $X_{k}$ of
$End({\mathfrak m}/{\mathfrak m}^{k+1})$ such that
$\alpha = {\rm exp}(X_{k})$. Given an algebraic
equation $H$ vanishing on $G_{k}$ we have that
$H({\rm exp}(t X_{k}))$ is a polynomial equation in $t$.
Since it vanishes in ${\mathbb Z}$ then it is identically $0$.
We obtain that ${\rm exp}(t X_{k})$ belongs to $G_{k}$ for any
$t \in {\mathbb C}$.
Then $\{ {\rm exp}(t X_{k}) \}$ is  a $1$-dimensional
algebraic subgroup  of $G_{k}$ containing $Id$ and $\alpha$.

The last property is an immediate consequence of the discussion in
the previous paragraph.
\end{proof}
\begin{lem}
\label{lem:elprox2}
Let $G$ be a subgroup of $\diffh{}{n}$.
Fix $k \in {\mathbb N}$. Consider connected components $B_{1}$,
$B_{2}$ of $G_{k+1}$ such that
$\pi_{k}(B_{1}) \cap \pi_{k}(B_{2}) \neq \emptyset$.
Then we obtain $B_{1}=B_{2}$.\end{lem}
\begin{proof}
Let $\beta_{1} \in B_{1}$ and $\beta_{2} \in B_{2}$ such that
$\pi_{k} (\beta_{1})=\pi_{k} (\beta_{2})$.
We obtain $\pi_{k} (\beta_{2} \circ \beta_{1}^{(-1)}) = Id$. Since
\[ \pi_{1} \circ \hdots \circ \pi_{k} ( \beta_{2} \circ \beta_{1}^{(-1)}) = Id \]
then $\beta_{2} \circ \beta_{1}^{(-1)}$ is a unipotent element of
$G_{k+1}$. Thus $\beta_{2} \circ \beta_{1}^{(-1)}$ is in the connected
component of the identity of $G_{k+1}$ by lemma \ref{lem:elpro}.
We deduce that $\beta_{1}$ and $\beta_{2}$ are in the same connected
component of $G_{k+1}$.
\end{proof}
\begin{lem}
\label{lem:elprox}
Let $G$ be a subgroup of $\diffh{}{n}$. Fix $k \in {\mathbb N}$.
Then we have $\pi_{k}(G_{k+1,0})=G_{k,0}$.
 In particular the number of connected
components of $G_{k+1}$ and $G_{k}$ coincide.
\end{lem}
\begin{proof}
The image $\pi_{k}(B)$ of a connected component of $G_{k+1}$ is an algebraic
subset of $G_{k}$. Since $\pi_{k}(G_{k+1})=G_{k}$
there exists one connected component $B_{0}$  of
$G_{k+1}$ such that $\pi_{k}(B_{0})=G_{k,0}$.
The equality $\pi_{k}(G_{k+1,0})=G_{k,0}$ is a consequence of
$Id \in \pi_{k}(G_{k+1,0}) \cap \pi_{k}(B_{0})$ and lemma \ref{lem:elprox2}.
Hence the image $\pi_{k}(B)$ of a connected component $B$ of $G_{k+1}$ is a
connected component of  $G_{k}$.
\end{proof}
\begin{cor}
\label{cor:elpro3}
Let $G$ be a connected subgroup of $\diffh{}{n}$.
Then
${G}_{k}$ is connected in the Zariski topology for any $k \in {\mathbb N}$.
\end{cor}
\begin{lem}
\label{lem:elpro2}
Let $G$ be a unipotent subgroup of $\diffh{}{n}$.
Then $G$ is connected and $\overline{G}^{(0)}$ is unipotent and connected.
\end{lem}
\begin{proof}
Fix $\varphi \in \diffh{u}{n}$.
Since $\varphi_{1}$ is unipotent then
 $\varphi_{k}$ is unipotent for any $k \in {\mathbb N}$.
We deduce that
$C_{k}$ is a group of unipotent matrices for any $k \in {\mathbb N}$.

Denote $m = \dim_{\mathbb C} {\mathfrak m}/{\mathfrak m}^{k+1}$.
An element $A \in GL({\mathfrak m}/{\mathfrak m}^{k+1})$ is unipotent
if and only if it satisfies the algebraic equation $(A-Id)^{m} \equiv 0$.
In particular $G_{k}$ is unipotent. We deduce that
$G_{k}$ is the connected component of the identity of $G_{k}$
for any $k \in {\mathbb N}$  by lemma \ref{lem:elpro}. Moreover since $G_{1}$
is unipotent then $\overline{G}^{(0)}$ is unipotent.
\end{proof}
So far we are working in the jet spaces. Next we interpret the
algebraic structure of the groups $G_{k}$ ($k \in {\mathbb N}$)
associated to $G \subset  \diffh{}{n}$ in terms of Lie subalgebras
of $\hat{\mathcal X} \cn{n}$.
\begin{defi}
Let $G$ a subgroup of $\diffh{}{n}$. We say that a complex Lie subalgebra
${\mathfrak h}$ of $\hat{\mathcal X} \cn{n}$ is the Lie algebra
of $G$ if
\[ {\mathfrak h}= \{ \hat{X} \in \hat{\mathcal X} \cn{n} : {\rm exp}(t \hat{X}) \in G \
\forall t \in {\mathbb C} \} . \]
\end{defi}
\begin{defi}
\label{def:liealg}
Let $G \subset \diffh{}{n}$ be a group.
We define
\[ {\mathfrak g} = \{ X \in \hat{\mathcal X} \cn{n} : X_{k} \in {\mathfrak g}_{k} \
\forall k \in {\mathbb N} \}  \]
where ${\mathfrak g}_{k}$ is the Lie algebra of the algebraic group $G_{k}$
(see equation (\ref{equ:actvf})).
Clearly ${\mathfrak g}$ is closed in the Krull topology.
\end{defi}
The Lie algebra of a group $G \subset \diffh{}{n}$ shares
the usual properties of Lie algebras of Lie groups.
\begin{pro}
\label{pro:lie}
Let $G \subset \diffh{}{n}$ be a group.
Then ${\mathfrak g}$ is the Lie algebra of  $\overline{G}^{(0)}$.
The group $\overline{G}_{0}^{(0)}$ is generated by ${\rm exp}({\mathfrak g})$.
Moreover if $G$ is unipotent then ${\mathfrak g}$ is a Lie algebra of formal
nilpotent vector fields and
${\rm exp}: {\mathfrak g} \to \overline{G}^{(0)}$
is a bijection.
\end{pro}
\begin{proof}
The definitions of ${\mathfrak g}$ (see def. \ref{def:liealg}) and
$\overline{G}^{(0)}$ (see def. \ref{def:liegrp}) imply that
${\mathfrak g}$ is the Lie algebra of $\overline{G}^{(0)}$.

We denote ${\mathfrak g}_{k}$ the Lie algebra of $G_{k}$.
There exists an homomorphism $\theta_{k}: {\mathfrak g}_{k+1} \to {\mathfrak g}_{k}$
that is defined analogously as $\pi_{k} : G_{k+1} \to G_{k}$.
Indeed $\theta_{k}$ is the differential of the map $\pi_{k}$ at $Id$.
Moreover since we work in characteristic $0$ the map $\theta_{k}$
is surjective (see prop. 7.2, page 105 \cite{Borel}).

Let $\varphi \in \overline{G}_{0}^{(0)}$. Since $G_{1,0}$ (see def. \ref{def:liegrp})
is connected
the set ${\rm exp}({\mathfrak g}_{1})$ generates $G_{1,0}$.
There exist $X_{1}^{1}$, $\hdots$, $X_{1}^{k}$ in ${\mathfrak g}_{1}$
such that
\[ \varphi_{1} = {\rm exp}(X_{1}^{1}) \circ \hdots \circ {\rm exp}(X_{1}^{k}) . \]
Given $X_{1}^{j}$ there exists $X^{j} \in {\mathfrak g}$ such that
$(X^{j})_{1} = X_{1}^{j}$ (see equation (\ref{equ:actvf})) for any $1 \leq j \leq k$.
It is a consequence of $\theta_{k}({\mathfrak g}_{k+1})={\mathfrak g}_{k}$ for any
$k \in {\mathbb N}$. Denote
\[ \eta  = {\rm exp}(X^{1}) \circ \hdots \circ {\rm exp}(X^{k}) . \]
The transformation $\varphi \circ \eta^{(-1)}$ has identity linear part
and so it is contained in  ${\rm exp}({\mathfrak g})$ by lemma \ref{lem:elpro}. Clearly
 $\varphi$ is a finite composition of elements in ${\rm exp}({\mathfrak g})$.

If $G$ is unipotent then $\overline{G}^{(0)}$ is unipotent by lemma \ref{lem:elpro2}.
Given an element $X$ of $\hat{\mathcal X} \cn{n} \cap {\mathfrak g}$ the derivation
$X_{1} \in {\mathfrak g}_{1}$ is nilpotent. Therefore
${\mathfrak g}$ is contained in $\hat{\mathcal X}_{N} \cn{n}$ and the mapping
${\rm exp}: {\mathfrak g} \to \overline{G}^{(0)}$ is bijective by
prop. \ref{pro:clas}.
\end{proof}
We associate $\overline{G}^{(0)}$ to a
given connected group $G \subset \diffh{}{n}$.
We want to study the derived series and the
descending central series of $\overline{G}^{(0)}$. A priori the elements
in these series are not necessarily closed in the Krull topology.
Anyway solvable
and nilpotent behavior can be interpreted at the jet level.
Thus replacing every element in the series with its closure in the
Krull topology provides a new series with essentially the same properties.
Moreover it is easier to handle since all the elements are Krull limits of
finite dimensional matrix Lie groups.
\begin{defi}
\label{def:clser}
Let $G \subset \diffh{}{n}$ be a group.
We define $\overline{G}^{(1)} $ the closed derived group of $G$ as
the closure of the derived group $[\overline{G}^{(0)},\overline{G}^{(0)}]$
in the Krull topology. We define the $(j+1)$-closed derived group
$\overline{G}^{(j+1)}$  of $G$ as the closure of
$[\overline{G}^{(j)},\overline{G}^{(j)}]$
for any $j \in {\mathbb N}$.
We define $\overline{\mathcal C}^{0} G = \overline{G}^{(0)}$ and
$\overline{\mathcal C}^{j+1} G$ as the closure of
$[\overline{\mathcal C}^{0} G,\overline{\mathcal C}^{j} G]$
for any $j \in {\mathbb N}$.
\end{defi}
So far we have associated a Lie algebra to any connected group. Next we
see that the properties of the derived series and the descending central series
for $\overline{G}^{(0)}$ and
${\mathfrak g}$ are analogous.   We introduce the
analogue of the closed derived series for Lie algebras.
\begin{defi}
Let $G \subset \diffh{}{n}$ be a group.
We define $\overline{\mathfrak g}^{(1)}$ the closed derived Lie algebra of
${\mathfrak g}$ as
the closure of the Lie algebra
${\mathfrak g}^{(1)}=[{\mathfrak g},{\mathfrak g}]$ in the Krull topology.
We define the $(j+1)$-closed derived Lie algebra $\overline{\mathfrak g}^{(j+1)}$
of ${\mathfrak g}$ as the closed derived Lie algebra of
 $\overline{\mathfrak g}^{(j)}$ for any $j \in {\mathbb N}$.
 Denote $\overline{\mathfrak g}^{(0)} = {\mathfrak g}$.
 We define $\overline{\mathcal C}^{0} {\mathfrak g} ={\mathfrak g}$ and
$\overline{\mathcal C}^{j+1} {\mathfrak g} $ as the closure of
$[\overline{\mathcal C}^{0} {\mathfrak g} ,\overline{\mathcal C}^{j} {\mathfrak g} ]$
for any $j \in {\mathbb N}$.
\end{defi}
Next lemmas are straightforward.
\begin{lem}
\label{lem:normalg}
Let $G \subset \diffh{}{n}$ be a group.
Then $\overline{G}^{(j)}$ (resp. $\overline{\mathcal C}^{j} G$)
is the closure in the Krull topology of
$(\overline{G}^{(0)})^{(j)}$ (resp.  ${\mathcal C}^{j} \overline{G}^{(0)}$)
for any $j \in {\mathbb N}$. Moreover
\[ \hdots \triangleleft \overline{G}^{(m)}
\triangleleft \hdots \triangleleft \overline{G}^{(1)} \triangleleft \overline{G}^{(0)}
\ \ {\rm and} \ \
 \hdots \triangleleft \overline{\mathcal C}^{m} G
\triangleleft \hdots \triangleleft \overline{\mathcal C}^{1} G \triangleleft
\overline{\mathcal C}^{0} G = \overline{G}^{(0)}  \]
are normal series.
\end{lem}
\begin{lem}
\label{lem:normal}
Let $G \subset \diffh{}{n}$ be a group.
We have that $\overline{\mathfrak g}^{(j)}$ (resp. $\overline{\mathcal C}^{j} {\mathfrak g}$)
is the closure in the Krull topology
of ${\mathfrak g}^{(j)}$ (resp. ${\mathcal C}^{j} {\mathfrak g}$)  for any $j \in {\mathbb N}$.
Moreover
\[ \hdots \triangleleft \overline{\mathfrak g}^{(m)}
\triangleleft \hdots \triangleleft \overline{\mathfrak g}^{(1)}
\triangleleft \overline{\mathfrak g}^{(0)} = {\mathfrak g}
\ \ {\rm and} \ \
\hdots \triangleleft \overline{\mathcal C}^{m} {\mathfrak g}
\triangleleft \hdots \triangleleft \overline{\mathcal C}^{1} {\mathfrak g}
\triangleleft \overline{\mathcal C}^{0} {\mathfrak g} = {\mathfrak g}  \]
are normal series.
\end{lem}
\begin{pro}
\label{pro:lieder}
Let $G \subset \diffh{}{n}$ be a connected group.
Then
\begin{itemize}
\item $\overline{\mathfrak g}^{(n)}$ is the Lie algebra of
$\overline{G}^{(n)} \ \forall n \in {\mathbb N}$.
 \item $\overline{\mathcal C}^{n} {\mathfrak g}$ is the Lie algebra of
$\overline{\mathcal C}^{n} G \ \forall n \in {\mathbb N}$.
\end{itemize}
\end{pro}
\begin{proof}
Let us prove the first property. The proof of the second property is analogous.

We denote ${\mathfrak g}_{k}$ the Lie algebra of $G_{k}$.
Since we work in characteristic $0$ and $G_{k}$ is a connected algebraic group
(cor. \ref{cor:elpro3})
we have that ${\mathfrak g}_{k}^{(j)}$ is the Lie algebra of $G_{k}^{(j)}$ for any
$j \in {\mathbb N}$ (prop. 7.8, page 108, \cite{Borel}).

Denote
\[ \tilde{G}^{(j)} = \{
\varphi \in \diffh{}{n} : \varphi_{k} \in G_{k}^{(j)} \ \forall k \in {\mathbb N} \}  \]
and
\[ \tilde{\mathfrak g}^{(j)} = \{
X \in \hat{\mathcal X} \cn{n} : X_{k} \in {\mathfrak g}_{k}^{(j)} \ \forall k \in {\mathbb N} \}.  \]
It is clear that $\tilde{\mathfrak g}^{(j)}$ is the Lie algebra of $\tilde{G}^{(j)}$
for any $j \in {\mathbb N} \cup \{0\}$. It suffices to prove that
$\overline{G}^{(j)} = \tilde{G}^{(j)}$ and
$\overline{\mathfrak g}^{(j)} =  \tilde{\mathfrak g}^{(j)}$  for any $j \in {\mathbb N} \cup \{0\}$.
It is obvious that $\overline{G}^{(j)} \subset \tilde{G}^{(j)}$ and
$\overline{\mathfrak g}^{(j)} \subset  \tilde{\mathfrak g}^{(j)}$ for any $j \in {\mathbb N} \cup \{0\}$.
Let us show the other contentions.

Denote $H = \overline{G}^{(0)}$.
We know that the morphism $\pi_{k} : G_{k+1} \to G_{k}$ is surjective for any
$k \in {\mathbb N}$.  Thus given $\alpha \in G_{k}$ there exists
$\varphi \in \overline{G}^{(0)}$ such that $\varphi_{k} = \alpha$.
As a consequence
$G_{k}^{(j)} = ({H}^{(j)})_{k}$ for all $k \geq 1$ and $j \geq 0$.
We deduce that an element of $\tilde{G}^{(j)}$ is limit in the Krull topology
of elements in ${H}^{(j)}$ for any $j \geq 0$.
Since the closure of $H^{(j)}$ in the Krull topology is $\overline{G}^{(j)}$
(lemma \ref{lem:normalg})
we obtain $\tilde{G}^{(j)} \subset \overline{G}^{(j)}$ for any $j \geq 0$.

Consider the homomorphism $\theta_{k}: {\mathfrak g}_{k+1} \to {\mathfrak g}_{k}$
defined in the proof of prop. \ref{pro:lie}.
Since $\theta_{k} ({\mathfrak g}_{k+1})= {\mathfrak g}_{k}$ for any  $k \in {\mathbb N}$
we can proceed analogously to show
$\tilde{\mathfrak g}^{(j)} \subset \overline{\mathfrak g}^{(j)}$ for any $j \geq 0$.
%
%
%
\end{proof}
\section{Soluble lengths of connected and unipotent groups}
\label{sec:sollen}
In this section we provide upper bounds for the soluble length of
connected, unipotent and nilpotent subgroups of
$\diffh{}{n}$. We can reduce the problem to an analogous but
simpler one on Lie algebras by using proposition \ref{pro:lieder}.
\begin{pro}
\label{pro:mainc}
Let $G \subset \diffh{}{n}$ be a connected solvable group.
Then we have $l(G) \leq 2n$.
\end{pro}
\begin{pro}
\label{pro:main1}
Let $G \subset \diffh{}{n}$ be a unipotent solvable group.
Then we have $l(G) \leq 2n-1$.
\end{pro}
\begin{pro}
\label{pro:main2}
Let $G \subset \diffh{}{n}$ be a nilpotent group.
Then we have $l(G) \leq n$.
\end{pro}
The optimality of the bounds, that completes the proof of theorems
\ref{teo:mainc},
\ref{teo:main} and \ref{teo:mainn}, is postponed for next section.

Proposition \ref{pro:mainc} is an immediate consequence of
propositions \ref{pro:lieder}, \ref{pro:lie} and next result.
\begin{pro}
\label{pro:alg}
Let ${\mathfrak g}$ a solvable Lie subalgebra of  $\hat{\mathcal X} \cn{n}$.
Then we have $l({\mathfrak g}) \leq 2n$.
\end{pro}
The next definitions are intended to provide the setting for the proof of prop.
\ref{pro:alg}.
\begin{defi}
We define $\hat{K}_{n}$ the field of fractions of the ring
of formal power series ${\mathbb C}[[x_{1},\hdots,x_{n}]]$.
We define $\overline{K}_{n}$ the algebraic closure of
$\hat{K}_{n}$.
\end{defi}
\begin{defi}
Let ${\mathfrak g}$ a Lie subalgebra of  $\hat{\mathcal X} \cn{n}$.
We define $\kappa (p)$
as the dimension of the ${\hat{K}_{n}}$ vector space
$\overline{\mathfrak g}^{(p)}  \otimes_{\mathbb C} \hat{K}_{n}$
for any $p \geq 0$.
It is obvious that $\kappa (p+1) \leq \kappa (p) \leq n$ for any $p \geq 0$.
\end{defi}
\begin{defi}
Let ${\mathfrak g}$ a Lie subalgebra of  $\hat{\mathcal X} \cn{n}$.
We define the field
\[ {\mathcal M}_{p} = \{ g \in \hat{K}_{n} : X(g) =0  \
\forall X \in \overline{\mathfrak g}^{(p)} \} \]
for any $p \geq 0$.
We denote $\overline{\mathcal M}_{p}$ the algebraic closure of
${\mathcal M}_{p}$.
\end{defi}
The proof of the inequality $l({\mathfrak g}) \leq 2n$ for a solvable Lie algebra
of formal vector fields
relies on the properties of the  sequence
\[ 0= \kappa (l({\mathfrak g})) \leq \kappa (l({\mathfrak g})-1) \leq \hdots \leq
\kappa(1) \leq \kappa(0) \leq n. \]
Indeed prop. \ref{pro:alg}  is a direct consequence of
proposition \ref{pro:intd}.
\begin{pro}
\label{pro:intd}
Let ${\mathfrak g}$ a solvable Lie subalgebra of  $\hat{\mathcal X} \cn{n}$.
Let $p \geq 0$ such that $\kappa(p)>0$.
Then we have $\kappa(p+2) < \kappa(p)$.
\end{pro}
The closed derived series
\[ \hdots \triangleleft \overline{\mathfrak g}^{(p+1)} \triangleleft \overline{\mathfrak g}^{(p)}
\triangleleft \overline{\mathfrak g}^{(p-1)} \triangleleft \overline{\mathfrak g}^{(p-2)}
\triangleleft
\hdots \triangleleft \overline{\mathfrak g}^{(1)}
\triangleleft \overline{\mathfrak g}^{(0)} = {\mathfrak g} \]
is normal.
The properties $[\overline{\mathfrak g}^{(r)}, \overline{\mathfrak g}^{(s)}] \subset
\overline{\mathfrak g}^{(r)}$ for $r >s$ and
$[\overline{\mathfrak g}^{(r)}, \overline{\mathfrak g}^{(r)}] \subset
\overline{\mathfrak g}^{(r+1)}$ for $r \geq 0$ impose restrictions on the elements of the
Lie algebras $\overline{\mathfrak g}^{(r)}$ ($r \geq 0$).
Our idea for the proof of propositions \ref{pro:maxd} and \ref{pro:intd} is associating
to any $\overline{\mathfrak g}^{(p)}$ a Lie algebra of matrices ${\mathcal G}_{p}$
with coefficients in ${\mathcal M}_{p}$. Then ${\mathcal G}_{p}$ is up to conjugation
a Lie algebra of upper triangular matrices by Lie's theorem \ref{teo:Lie}.
The eigenvalues are elements of
$\overline{\mathcal M}_{p}$. If all the eigenvalues of all elements of ${\mathcal G}_{p}$
are $0$ it is easy to see that $\kappa(p+1) < \kappa (p)$.
Otherwise all the eigenvalues of elements of
$[{\mathcal G}_{p} , {\mathcal G}_{p}]$ are $0$ and we obtain $\kappa (p+2) < \kappa (p)$.

Let us fix the setup for the rest of the section.
Let ${\mathfrak g} \subset \hat{\mathcal X} \cn{n}$ be a solvable Lie algebra.
Fix $p \in {\mathbb N}$ such that
$\kappa(p+1) < \kappa (p)$.
Denote $q = \kappa (p+1)$ and $m=\kappa (p) - \kappa (p+1)$.
There exists a base $\{ Y_{1}, \hdots, Y_{q} \}$ contained in
$\overline{\mathfrak g}^{(p+1)}$
of the vector space
$\overline{\mathfrak g}^{(p+1)}  \otimes_{\mathbb C} \hat{K}_{n}$.
Moreover there exists a base
$\{ Y_{1}, \hdots, Y_{q}, X_{1}, \hdots, X_{m}\}$ contained in
$\overline{\mathfrak g}^{(p)}$
of the vector space
$\overline{\mathfrak g}^{(p)}  \otimes_{\mathbb C} \hat{K}_{n}$.

Any element $Z \in \overline{\mathfrak g}^{(p)}$ is of the form
\begin{equation}
\label{equ:dev}
 Z = b_{1} Y_{1} + \hdots b_{q} Y_{q} + a_{1} X_{1} + \hdots + a_{m} X_{m}
\end{equation}
where $b_{1}, \hdots b_{q}, a_{1}, \hdots, a_{m} \in \hat{K}_{n}$.
We have
\begin{equation}
\label{equ:bra1}
 [Y_{j}, Z] = \sum_{k=1}^{q} Y_{j} (b_{k}) Y_{k} + \sum_{k=1}^{m} Y_{j} (a_{k}) X_{k} +
 \sum_{k=1}^{q} b_{k} [Y_{j}, Y_{k}] + \sum_{k=1}^{m} a_{k} [Y_{j}, X_{k}].
 \end{equation}
 Since
 $[ \overline{\mathfrak g}^{(p+1)},  \overline{\mathfrak g}^{(p)}] \subset
 \overline{\mathfrak g}^{(p+1)}$ (lemma \ref{lem:normal}) we obtain
 $Y_{j}(a_{k})=0$ for all $1 \leq j \leq q$ and $1 \leq k \leq m$.
 An analogous argument using the formula
 \begin{equation}
 \label{equ:bra2}
  [X_{j}, Z] = \sum_{k=1}^{q} X_{j} (b_{k}) Y_{k} + \sum_{k=1}^{m} X_{j} (a_{k}) X_{k} +
 \sum_{k=1}^{q} b_{k} [X_{j}, Y_{k}] + \sum_{k=1}^{m} a_{k} [X_{j}, X_{k}]
 \end{equation}
and  $[ \overline{\mathfrak g}^{(p)},  \overline{\mathfrak g}^{(p)}] \subset
 \overline{\mathfrak g}^{(p+1)}$
provide
$X_{j}(a_{k})=0$ for all $1 \leq j \leq m$ and $1 \leq k \leq m$.

The next lemma is a consequence of the previous discussion.
\begin{lem}
\label{lem:aux}
Suppose $\kappa(p+1) < \kappa (p)$.
Let $Z = \sum_{j=1}^{q} b_{j} Y_{j} + \sum_{k=1}^{m} a_{k} X_{k}$
be an element of $\overline{\mathfrak g}^{(p)}$. Then
$a_{1}, \hdots , a_{m}$ belong to ${\mathcal M}_{p}$.
\end{lem}
\begin{lem}
\label{lem:aux2}
Suppose $\kappa(r) = \kappa (p)$ for some $r < p$.
Consider the notations in this section.
Let  $Z = \sum_{j=1}^{q} b_{j} Y_{j} + \sum_{k=1}^{m} a_{k} X_{k}$
be an element of $\overline{\mathfrak g}^{(r)}$. Then
$a_{k}$ belongs to ${\mathcal M}_{p+1}$ for any $1 \leq k \leq m$
and
$X_{j}(a_{k}) \in {\mathcal M}_{p}$ for all $1 \leq j \leq m$ and $1 \leq k \leq m$
\end{lem}
\begin{proof}
It is clear that
$\{ Y_{1}, \hdots, Y_{q}, X_{1}, \hdots, X_{m}\} $ is a base of
$\overline{\mathfrak g}^{(r)}  \otimes_{\mathbb C} \hat{K}_{n}$.
Any $Z \in \overline{\mathfrak g}^{(r)}$ is of the form
(\ref{equ:dev}). Then
$[Y_{j},Z] \in [\overline{\mathfrak g}^{(p+1)}, \overline{\mathfrak g}^{(r)}]
\subset \overline{\mathfrak g}^{(p+1)}$ implies
$Y_{j}(a_{k})=0$ for $1 \leq j \leq q$ and $1 \leq k \leq m$.
The element
$X_{j}(a_{k})$ is the coefficient of $X_{k}$ of
$[X_{j},Z] \in [\overline{\mathfrak g}^{(p)}, \overline{\mathfrak g}^{(r)}]
\subset \overline{\mathfrak g}^{(p)}$.
Thus $X_{j}(a_{k}) \in {\mathcal M}_{p}$ for all $1 \leq j \leq m$ and $1 \leq k \leq m$
by lemma \ref{lem:aux}.
\end{proof}
Suppose $\kappa(r) = \kappa (p)$ for some $r < p$.
We associate to any formal vector field
$Z = \sum_{j=1}^{q} b_{j} Y_{j} + \sum_{k=1}^{m} a_{k} X_{k}
\in  \overline{\mathfrak g}^{(r)}$ the matrix
\begin{equation}
\label{equ:matrix}
M_{Z} =
\left(
\begin{array}{cccc}
X_{1}(a_{1}) & X_{1}(a_{2}) & \hdots & X_{1}(a_{m}) \\
X_{2}(a_{1}) & X_{2}(a_{2}) & \hdots & X_{2}(a_{m}) \\
\vdots & \vdots & & \vdots \\
X_{m}(a_{1}) & X_{m}(a_{2}) & \hdots & X_{m}(a_{m})
\end{array}
\right)
\end{equation}
with coefficients in ${\mathcal M}_{p}$. In this way we can associate
a solvable Lie algebra of matrices (in which we can apply Lie's theorem
\ref{teo:Lie}) to $\overline{\mathfrak g}^{(r)}$.
\begin{lem}
\label{lem:sol}
Suppose $\kappa(r) = \kappa (p)$ for some $r < p$. Then
we have $M_{[Z,W]} = [M_{Z}, M_{W}]= M_{Z} M_{W} - M_{W} M_{Z}$ for all
 $Z, W \in  \overline{\mathfrak g}^{(r)}$.
\end{lem}
\begin{proof}
Let
\[
Z = \sum_{j=1}^{q} b_{j} Y_{j} +\sum_{k=1}^{m} a_{k} X_{k}  \ {\rm and} \
W = \sum_{j=1}^{q} \beta_{j} Y_{j} +\sum_{k=1}^{m} \alpha_{k} X_{k} . \]
We have
\[ [Z,W] =  \sum_{j=1}^{q} c_{j} Y_{j} +
\sum_{k=1}^{m} (Z(\alpha_{k}) - W(a_{k})) X_{k} \]
for some $c_{1},\hdots,c_{q} \in \hat{K}_{n}$. We obtain
\[ \sum_{k=1}^{m} (Z(\alpha_{k}) - W(a_{k})) X_{k}  =
\sum_{k=1}^{m} \left[
(\sum_{j=1}^{m} a_{j}X_{j})(\alpha_{k}) - (\sum_{j=1}^{m} \alpha_{j}X_{j})(a_{k})
\right] X_{k} \]
since $a_{k} \in {\mathcal M}_{p+1} \ni \alpha_{k}$ for any $1 \leq k \leq m$ (lemma \ref{lem:aux2}).
The element $(M_{[Z,W]})_{jk}$ satisfies
\[ (M_{[Z,W]})_{jk} = X_{j} \left(
 (\sum_{d=1}^{m} a_{d}X_{d})(\alpha_{k}) - (\sum_{d=1}^{m} \alpha_{d} X_{d})(a_{k})
\right)  \]
and then
\[ (M_{[Z,W]})_{jk} =
 \sum_{d=1}^{m} X_{j}(a_{d}) X_{d}(\alpha_{k}) -
 \sum_{d=1}^{m} X_{j}(\alpha_{d})X_{d}(a_{k})  \]
 since $X_{d}(a_{k}) \in {\mathcal M}_{p} \ni X_{d}(\alpha_{k})$ for $1 \leq k \leq m$.
\end{proof}
A formal vector field $X \in \hat{\mathcal X} \cn{n}$ is a derivation of the field
$\hat{K}_{n}$.
It can be extended to the field $\overline{K}_{n}$.
\begin{pro}
\label{pro:ext}
(see \cite{Bourbaki.alg.47} A.V.129, prop. 4)
Let  $X \in \hat{\mathcal X} \cn{n}$. Then $X$ induces a unique derivation in
$\overline{K}_{n}$.  Moreover it satisfies  $X_{| \overline{\mathcal M}_{p}} \equiv 0$
if  $X_{|{\mathcal M}_{p}} \equiv 0$.
\end{pro}
We need the extension of $X$ to $\overline{K}_{n}$ since we plan to apply next theorem
to Lie algebras of matrices of the form (\ref{equ:matrix})
and ${\mathcal M}_{p}$ is not algebraically closed.
\begin{teo}
\label{teo:Lie}
(Lie, see \cite{Serre.Lie}, chapter V, section 5, page 36).
Let ${\mathfrak g}$ be a solvable Lie algebra over an algebraically closed field
$K$ of characteristic $0$. Let $\varrho$ be a linear representation of
${\mathfrak g}$ with representation space $V$. Then up to change of base in $V$
the Lie algebra $\varrho({\mathfrak g})$ is composed of upper triangular matrices.
\end{teo}
\begin{proof}[proof of prop. \ref{pro:intd}]
It suffices to prove that
$\kappa(p+1) < \kappa(p) = \kappa(p-2)$ is impossible.
Consider the notations in this section.

We define the complex Lie algebra
\[ {\mathcal G} = \{ M_{Z} : Z \in \overline{\mathfrak g}^{(p-2)} \} . \]
It satisfies ${\mathcal G}^{(2)} = 0$ by lemmas  \ref{lem:sol} and \ref{lem:aux}.
We define
$\overline{\mathcal G} = {\mathcal G} \otimes_{\mathbb C} \overline{\mathcal M}_{p}$,
it is a Lie algebra over the field $\overline{\mathcal M}_{p}$.
Since $(\overline{\mathcal G})^{(2)}=0$ then $\overline{\mathcal G}$ is solvable.

The Lie theorem (th. \ref{teo:Lie}) implies that there exists
$N \in GL(n,\overline{\mathcal M}_{p})$
such that $N^{-1} M N$ is an upper triangular matrix for any
$M \in {\mathcal G}$. We define the derivations
$W_{1}, \hdots, W_{m}$ of $\overline{K}_{n}$ by the formula
\begin{equation}
\label{equ:linalg}
(W_{1}, \hdots, W_{m}) = (X_{1},\hdots,X_{m}) N^{T} .
\end{equation}
Any $Z \in \overline{\mathfrak g}^{(p-2)}$ is of the form
\[ Z= b_{1} Y_{1}+ \hdots + b_{q} Y_{q} + \alpha_{1} W_{1} + \hdots + \alpha_{m} W_{m} \]
where $b_{1},\hdots,b_{q} \in \hat{K}_{n}$ and
$\alpha_{1},\hdots,\alpha_{m} \in  {\mathcal M}_{p+1} \otimes_{\mathbb C} \overline{\mathcal M}_{p}$.
The matrix
\[ \tilde{M}_{Z} =
\left(
\begin{array}{cccc}
W_{1}(\alpha_{1}) & W_{1}(\alpha_{2}) & \hdots & W_{1}(\alpha_{m}) \\
W_{2}(\alpha_{1}) & W_{2}(\alpha_{2}) & \hdots & W_{2}(\alpha_{m}) \\
\vdots & \vdots & & \vdots \\
W_{m}(\alpha_{1}) & W_{m}(\alpha_{2}) & \hdots & W_{m}(\alpha_{m})
\end{array}
\right)
\]
satisfies $\tilde{M}_{Z} = N M_{Z} N^{-1}$ (here it is key that the elements
of $\overline{\mathcal M}_{p}$ are first integrals of the vector fields in
$\overline{\mathfrak g}^{(p-2)}$ by prop. \ref{pro:ext}).
It is an upper triangular matrix.

Since $\tilde{M}_{[Z,T]} = [\tilde{M}_{Z}, \tilde{M}_{T}]$ for all
$Z,T \in \overline{\mathfrak g}^{(p-2)}$ we obtain
that $\tilde{M}_{Z}$ is an upper triangular matrix with
vanishing principal diagonal for any $Z \in (\overline{\mathfrak g}^{(p-2)})^{(1)}$.
Indeed any element $Z$ of $(\overline{\mathfrak g}^{(p-2)})^{(1)}$ is of the form
\[ Z= b_{1} Y_{1}+ \hdots + b_{q} Y_{q} + \alpha_{1} W_{1} + \hdots + \alpha_{m} W_{m} \]
with $W_{1}(\alpha_{1})=W_{2}(\alpha_{1})=\hdots=W_{m}(\alpha_{1})=0$
and $\alpha_{1} \in {\mathcal M}_{p+1} \otimes_{\mathbb C} \overline{\mathcal M}_{p}$.
We obtain that
\[ [Z,T] =  \tilde{b}_{1} Y_{1}+ \hdots + \tilde{b}_{q} Y_{q} +
\tilde{\alpha}_{2} W_{2} + \hdots +\tilde{\alpha}_{m} W_{m} \]
for all $Z,T \in (\overline{\mathfrak g}^{(p-2)})^{(1)}$.
This leads us to
\[ \dim_{\overline{K}_{n}} [(\overline{\mathfrak g}^{(p-2)})^{(2)} \otimes_{\mathbb C}
\overline{K}_{n}] < \kappa (p). \]
We obtain
\[  \dim_{\hat{K}_{n}} [(\overline{\mathfrak g}^{(p-2)})^{(2)} \otimes_{\mathbb C}
\hat{K}_{n}] =
 \dim_{\overline{K}_{n}} [(\overline{\mathfrak g}^{(p-2)})^{(2)} \otimes_{\mathbb C}
\overline{K}_{n}] < \kappa (p).
\]
The algebra $\overline{\mathfrak g}^{(p)}$ is the closure of
$(\overline{\mathfrak g}^{(p-2)})^{(2)}$ in the Krull topology (lemma \ref{lem:normal}).
The formula
\[  \kappa(p) = dim_{\hat{K}_{n}} (\overline{\mathfrak g}^{(p)} \otimes_{\mathbb C}
\hat{K}_{n}) =
dim_{\hat{K}_{n}} [(\overline{\mathfrak g}^{(p-2)})^{(2)} \otimes_{\mathbb C}
\hat{K}_{n}] < \kappa(p) \]
represents a contradiction.
\end{proof}
Our next goal is proving prop. \ref{pro:main1}.
The only difference with the general connected case is that
$\kappa (0)=n$ implies $\kappa(1) < n$.
\begin{lem}
\label{lem:eigen}
Let $X \in \hat{\mathcal X}_{N} \cn{n}$. Suppose that there exist
$\lambda \in {\mathbb C}$ and $0 \neq h  \in \hat{K}_{n}$ such that
$X(h)=\lambda h$. Then we obtain $\lambda = 0$.
\end{lem}
\begin{proof}
Let $h = f/g$ where $f,g \in {\mathbb C}[[x_{1},\hdots,x_{n}]]$
do not have common factors.
Consider  $f = f_{a} + f_{a+1} + \hdots$ and
$g=g_{b} + g_{b+1} + \hdots$ the decompositions of $f$, $g$ in homogeneous components
where $f_{a} \not \equiv 0 \not \equiv g_{b}$.
We obtain $(j^{1} X)(f_{a}/g_{b})= \lambda f_{a}/g_{b}$.
A priori $f_{a}$ and $g_{b}$ can share common factors. Anyway
we can express $f_{a}/g_{b}$ in the form $\alpha / \beta$ where
$\alpha$ and $\beta$ are coprime homogeneous polynomials. Denote $Y = j^{1} X$. We have
\[ Y \left( \frac{\alpha}{\beta} \right) = \lambda \frac{\alpha}{\beta} \Leftrightarrow
Y(\alpha) \beta - \alpha Y(\beta) = \lambda \alpha \beta. \]
The previous equation implies that since $Y(\alpha)$ and $\alpha$ have the same degree
and $\alpha$ divides $Y(\alpha)$ then $Y(\alpha)=\mu \alpha$ for some $\mu \in {\mathbb C}$.
Analogously we obtain $Y(\beta) = \rho \beta$ for some
$\rho \in {\mathbb C}$. Since $Y$ is represented by a
nilpotent matrix then $Y^{n} \equiv 0$. This leads us to
\[ \rho^{n} \beta = Y^{n} (\beta) = 0 = Y^{n} (\alpha) = \mu^{n} \alpha \implies \mu=\rho=0 . \]
Clearly this implies $\lambda=0$.
\end{proof}
\begin{pro}
\label{pro:maxd}
Let ${\mathfrak g}$ a solvable Lie subalgebra of  $\hat{\mathcal X}_{N} \cn{n}$.
Let $p \geq 0$ such that $\kappa(p)>0$. Suppose that
${\mathcal M}_{p} = {\mathbb C}$.
Then we have $\kappa(p+1) < \kappa(p)$.
\end{pro}
\begin{proof}
It suffices to prove that the properties ${\mathcal M}_{p} = {\mathbb C}$,
$\kappa(p+1) < \kappa(p)$ and
$\kappa(p) = \kappa(p-1)$  are contradictory.
We use the notations in the proof of prop. \ref{pro:intd}.

An element $Z = \sum_{j=1}^{q} b_{j} Y_{j} +\sum_{k=1}^{m} a_{k} X_{k}$
of $\overline{\mathfrak g}^{(p-1)}$ induces a linear mapping
in the $m$ dimensional complex vector space $<a_{1},\hdots,a_{m}>$ whose matrix is ${M}_{Z}$.
The matrix $\tilde{M}_{Z}$ is upper triangular for
$Z \in \overline{\mathfrak g}^{(p-1)}$. The elements in the diagonal are
complex eigenvalues of
$Z $ interpreted as a linear mapping in
the complex space $<a_{1},\hdots,a_{m}>$.
But all such eigenvalues are zero (lemma \ref{lem:eigen}).
Proceeding as in the proof of prop. \ref{pro:intd} we obtain a contradiction.
\end{proof}
\begin{cor}
\label{cor:maxd}
Let ${\mathfrak g}$ a solvable Lie subalgebra of  $\hat{\mathcal X}_{N} \cn{n}$
such that $\kappa (0)=n$. Then we have $\kappa(1) < n$.
\end{cor}
\begin{proof}[proof of prop. \ref{pro:main1}]
The Lie algebra ${\mathfrak g}$ associated to $\overline{G}^{(0)}$
is contained in $\hat{\mathcal X}_{N} \cn{n}$ by proposition \ref{pro:lie}.
It is solvable (prop. \ref{pro:lieder}).
We obtain $l(G) = l({\mathfrak g}) \leq 2n-1$ by corollary \ref{cor:maxd}
and proposition \ref{pro:intd}.
\end{proof}
\begin{rem}
\label{rem:ing}
Let $G$ be a nilpotent subgroup of $GL(n,K)$ with $K$ algebraically closed.
Then $G$ is semisimple (i.e. every element of $G$ is diagonalizable) if and only
if $G$ is completely reducible (th. 7.7, page 94 \cite{Wehrfritz}).
There exists a classification of maximal irreducible locally nilpotent subgroups
of $GL(n,K)$ (Suprunenko \cite{Suprunenko}).  It implies
$l(G) \leq 1 + \max(m_{1},\hdots,m_{a})$ where
$n=p_{1}^{m_{1}} \hdots p_{a}^{m_{a}}$ is the prime factorization.
In particular $l(G) \leq [\log_{2} n] +1$ for any nilpotent subgroup of
$GL(n,K)$ and the bound is optimal.
This remark is used in the proof of prop. \ref{pro:main2}.
\end{rem}
\begin{proof}[proof of prop. \ref{pro:main2}]
Let us suppose that $G$ is connected (and of course nilpotent).
It suffices to prove that
the sequence $\{ \kappa (j) \}$ associated to its Lie algebra ${\mathfrak g}$
is strictly decreasing.
We claim that $\kappa(p+1) < \kappa(p)$ and
$\kappa(p) = \kappa(p-1)$  are not compatible.
We use the notations in the proof of prop. \ref{pro:intd}.

Any element $Z$ of $\overline{\mathfrak g}^{(p-1)}$ is of the form
\[ Z= b_{1} Y_{1}+ \hdots + b_{q} Y_{q} + \alpha_{1} W_{1} + \hdots + \alpha_{m} W_{m} \]
with $W_{2}(\alpha_{1})=\hdots=W_{m}(\alpha_{1})=0$
and $\alpha_{1} \in {\mathcal M}_{p+1} \otimes_{\mathbb C} \overline{\mathcal M}_{p}$
since $\tilde{M}_{Z}$ is upper triangular.
If $W_{1}(\alpha_{1})=0$ for any choice of $Z$ we reproduce the proof
of prop. \ref{pro:intd} to obtain a contradiction.
From now on we fix an element $Z$ of $\overline{\mathfrak g}^{(p-1)}$
such that $W_{1}(\alpha_{1}) \neq 0$.

Since $X_{1}, \hdots, X_{m} \in \overline{\mathfrak g}^{(p)}$ there exists an element
\[ Z'= \beta_{1} Y_{1}+ \hdots + \beta_{q} Y_{q} + a_{1} W_{1} + \hdots + a_{m} W_{m}
\in \overline{\mathfrak g}^{(p)} \]
such that $a_{1}, \hdots, a_{m} \in \overline{\mathcal M}_{p}$ and
$a_{1} \neq 0$ by equation (\ref{equ:linalg}).
The coefficient in $W_{1}$ of $A_{1}=[Z,Z']$ is equal to $-a_{1} W_{1}(\alpha_{1})$.
The coefficient in $W_{1}$ of $A_{2}=[Z,[Z,Z']]$ is equal to $a_{1} W_{1}(\alpha_{1})^{2}$.
We denote $A_{j+1} = [Z, A_{j}]$. The coefficient in $W_{1}$ of $A_{j}$
is equal to $(-1)^ {j} a_{1} W_{1}(\alpha_{1})^{j}$ for any $j \in {\mathbb N}$.
As a consequence $\overline{\mathfrak g}^{(p-1)}$ is not nilpotent. This contradicts
prop. \ref{pro:lieder}.

Now let us consider a general nilpotent subgroup $G$ of $\diffh{}{n}$.
Since $G_{k}$ is a nilpotent linear group the sets $G_{k,s}$ and $G_{k,u}$ are subgroups of
$G_{k}$ by a theorem of Suprunenko and Ty{\v{s}}kevi{\v{c}}
(see th. 7.11, page 97  \cite{Wehrfritz}).
We deduce that $\overline{G}_{s}^{(0)}$ and $\overline{G}_{u}^{(0)}$
are subgroups of $\overline{G}^{(0)}$.
The semisimple and unipotent parts of a linear algebraic nilpotent group commute
(lemma 7.10, page 96 \cite{Wehrfritz}). We obtain
$[G_{k,s},G_{k,u}] = \{Id\}$ for any $k \in {\mathbb N}$ and then
$[\overline{G}_{s}^{(0)},\overline{G}_{u}^{(0)}] = \{Id\}$.

The mapping $\overline{G}_{s}^{(0)} \to {G}_{1,s}$ that maps a diffeomorphism
to its linear part is an isomorphism of groups.
It is easy to check out  the inequality $l({G}_{1,s}) \leq n$
by using remark \ref{rem:ing}.
The soluble length satisfies
\[ l(G) = l(\overline{G}^{(0)}) = \max (l(\overline{G}_{s}^{(0)}), l (\overline{G}_{u}^{(0)}))
=  \max (l({G}_{1,s})), l (\overline{G}_{u}^{(0)}))  \leq n  \]
by the first part of the proof.
\end{proof}
\begin{rem}
Lie algebras of local vector fields and groups of germs of diffeomorphisms behave
differently than finite dimensional Lie algebras and linear algebraic groups.
For instance a finite dimensional Lie algebra ${\mathfrak g}$ over a field of characteristic
$0$ has a nilpotent first commutator group ${\mathfrak g}^{(1)}$ (see
\cite{Serre.Lie}, pag. 37). The analogue for Lie algebras of formal vector fields
 would imply that $G^{(1)}$ would be nilpotent for any
solvable connected subgroup $G \subset \diffh{}{n}$.
This is false since then the soluble length $l(G)$
of a connected solvable group $G \subset \diffh{}{n}$
would be less or equal than $n+1$ by th.  \ref{teo:mainn} contradicting
theorems \ref{teo:mainc} and \ref{teo:main}.
%
%
\end{rem}
%
%
\section{Examples}
\label{sec:exa}
The bounds provided in propositions \ref{pro:mainc}, \ref{pro:main1}
and \ref{pro:main2} for soluble lengths are optimal. We provide
specific examples of solvable groups with maximum length.

Fix $n \in {\mathbb N}$.
Let ${\mathfrak m}$ the maximal ideal of ${\mathbb C}[[x_{1},\hdots,x_{n}]]$.
We define the vector spaces
\[  U_{j} = \left\{  a(x_{j+1},\hdots,x_{n}) \frac{\partial}{\partial x_{j}} :
a \in {\mathfrak m}^{2} \right\} \ {\rm and} \
U_{n} = \left\{ t x_{n}^{2} \frac{\partial}{\partial x_{n}} : t \in {\mathbb C}
\right\} \]
for any $1 \leq j \leq n-1$. We define
\[ V_{j} = \left\{  x_{j} b(x_{j+1},\hdots,x_{n}) \frac{\partial}{\partial x_{j}} :
b \in {\mathfrak m} \right\}  \ {\rm and} \
V_{n} = \left\{ t x_{n} \frac{\partial}{\partial x_{n}} : t \in {\mathbb C}
\right\} \]
for any $1 \leq j \leq n-1$.
We define ${\mathcal G}_{2n}=0$ and
\[ {\mathcal G}_{2n-2j}= U_{1} \oplus V_{1} \oplus \hdots \oplus U_{j} \oplus V_{j},
\ {\mathcal G}_{2n-(2j+1)}=    {\mathcal G}_{2n-2j} \oplus U_{j+1} . \]
for $0 \leq j \leq n-1$.
All the formal vector fields in the vector spaces
${\mathcal G}_{1}, \hdots, {\mathcal G}_{2n-1}, {\mathcal G}_{2n}$
are nilpotent since they have vanishing first jets.
The proof of next lemma is straightforward.
\begin{lem}
We have
\begin{itemize}
\item $[U_{j},U_{j}]=0$ for any $1 \leq j \leq n$
\item $x_{n}^{2} U_{j} \subset [U_{j},U_{k}] \subset U_{j}$ for any $1 \leq j < k \leq n$.
\item $[V_{j},V_{j}]=0$ for any $1 \leq j \leq n$
\item $x_{k} x_{n} V_{j} \subset [V_{j},V_{k}] \subset V_{j}$ for any $1 \leq j < k \leq n$.
\item  $x_{n} U_{j} \subset [U_{j},V_{j}] \subset U_{j}$ for any $1 \leq j \leq n-1$
and $[U_{n},V_{n}]=U_{n}$.
\item  $x_{k} x_{n} U_{j} \subset [U_{j},V_{k}] \subset U_{j}$ for any $1 \leq j < k \leq n$.
\item $x_{n}^{2} V_{j} \subset [V_{j},U_{k}] \subset V_{j}$ for any $1 \leq j < k \leq n$.
\end{itemize}
\end{lem}
\begin{cor}
\label{cor:res}
We have $x_{n}^{2} {\mathcal G}_{2} + U_{n} \subset {\mathcal G}_{0}^{(1)}$ and
$x_{n}^{2}  {\mathcal G}_{j+1} \subset {\mathcal G}_{j}^{(1)} \subset {\mathcal G}_{j+1}$
for any $1 \leq j \leq 2n-1$.
Then $x_{n}^{c_{j}} {\mathcal G}_{j} \subset {\mathcal G}_{0}^{(j)} \subset  {\mathcal G}_{j}$
where $c_{j} = 2^{j} + 2^ {j-2}-2$ for any  $2 \leq j \leq 2n$.
Moreover the length $l({\mathcal G}_{0})$ of ${\mathcal G}_{0}$ is equal to $2n$.
\end{cor}
 \begin{proof}
 The property ${\mathcal G}_{j}^{(1)} \subset {\mathcal G}_{j+1}$
for any $0 \leq j \leq 2n-1$ is an inmediate consequence of the construction.
Hence we obtain $l({\mathcal G}_{0}) \leq 2n$.

  We have
 \[  U_{n} = [U_{n},V_{n}] ,  \
x_{n}^{2} U_{j} \subset  [U_{j},V_{n}] \ {\rm and} \  \ x_{n}^{2} V_{j}  \subset [V_{j},V_{n}] \]
 for any $1 \leq j  < n$.
Thus ${\mathcal G}_{0}^{(1)}$ contains $x_{n}^{2} {\mathcal G}_{2} + U_{n}$.
Then $l({\mathcal G}_{0})$ is equal to $2$ if $n=1$.

 Given $1 \leq k \leq n$ we have
 \[ x_{n}^{2} U_{j} \subset [U_{j},U_{k}] \ {\rm and} \ x_{n}^{2} V_{j} \subset [V_{j},U_{k}] \]
 for any $1 \leq j  < k$.
 Thus we obtain that ${\mathcal G}_{0}^{(2)}$ contains $x_{n}^{3} {\mathcal G}_{2}$.
 Moreover
 ${\mathcal G}_{2n-(2k-1)}^{(1)}$ contains $x_{n}^{2} {\mathcal G}_{2n-2(k-1)}$
 for any $1 \leq k \leq n$.

   Given $1 \leq k \leq n-1$ we have
 \[ x_{n} U_{k} \subset [U_{k},V_{k}] ,  \
x_{n}^{2} U_{j} \subset  [U_{j},U_{k}] \ {\rm and} \  \ x_{n}^{2} V_{j}  \subset [V_{j},U_{k}] \]
 for any $1 \leq j  < k$.
Thus ${\mathcal G}_{2n-2k}^{(1)}$ contains $x_{n}^{2} {\mathcal G}_{2n-2k+1}$
 for any $1 \leq k \leq n$. This leads us to
 \[ x_{n}^ {3}  {\mathcal G}_{2} \subset {\mathcal G}_{0}^{(2)} \implies
 x_{n}^ {2 \cdot 3 +2}  {\mathcal G}_{3} \subset {\mathcal G}_{0}^{(3)}  \implies
 \hdots \implies  x_{n}^ {c_{2n-1}}  {\mathcal G}_{2n-1} \subset {\mathcal G}_{0}^{(2n-1)}  \]
 for $n \geq 2$. We deduce that $l({\mathcal G}_{0})$ is equal to $2n$.
\end{proof}
\begin{rem}
The Lie algebra ${\mathcal G}_{0} \cap {\mathcal X} \cn{n}$ has soluble length $2n$.
It provides the example completing the proof of theorem \ref{teo:mainca}.
\end{rem}
\begin{defi}
We define the subgroup $H_{0}$ of $\diffh{}{n}$ of diffeomorphisms of the form
\[
(a_{1}(x_{2},\hdots,x_{n}) + x_{1} b_{1}(x_{2},\hdots,x_{n}), \hdots,
a_{n-1}(x_{n}) + x_{n-1} b_{n-1}(x_{n}), h^{\lambda,\mu}(x_{n}) ) \]
where $h^{\lambda,\mu}(x_{n})= \lambda x_{n} / (1 + \mu x)$,
$\lambda \in {\mathbb C}^{*}$, $\mu \in {\mathbb C}$,
$b_{j} -1 \in {\mathfrak m}$ and $a_{j} \in {\mathfrak m}^{2}$
for any $1 \leq j \leq n-1$. Let us notice that the families
${\{ h^{\lambda,\mu} \}}_{(\lambda,\mu) \in {\mathbb C}^{*} \times {\mathbb C}}$
and
${\{ h_{s,t} \}}_{(s,t) \in {\mathbb C} \times {\mathbb C}}$ coincide
where $h_{s,t}(x_{n})={\rm exp}((s x_{n} +t x_{n}^{2}) \partial /\partial x_{n})$.
\end{defi}
\begin{pro}
\label{pro:exa}
Fix $n \in {\mathbb N}$. Then
$H_{0}$ is a connected solvable group contained in  $\diffh{}{n}$ with
$l(H_{0})=2n$. Denote $G_{0} = H_{0} \cap \diff{}{n}$. Then $G_{0}$ is a
connected solvable group contained in $\diff{}{n}$ such that $l(G_{0})=2n$.
\end{pro}
\begin{proof}
It is clear from the construction that in every jet space the group $(H_{0})_{k}$
is algebraic. The group $H_{0}$ is closed in the Krull topology, we obtain
$\overline{H_{0}}^{(0)}= H_{0}$.
Moreover $H_{0}$ is connected since $j^{1} H_{0} \sim {\mathbb C}^{*}$

We want to calculate the Lie algebra ${\mathfrak g}$ associated to $H_{0}$.
We obtain ${\mathcal G}_{0} \subset {\mathfrak g}$ by using
the exponential formula (\ref{equ:exp}).
Given $1 \leq j \leq n-1$ and $X \in {\mathfrak g}$ we have that
$(x_{j} \circ {\rm exp}(t X) - x_{j})/t$ is of the form
\[ a_{j}(x_{j+1},\hdots,x_{n})  + x_{j} b_{j}(x_{j+1},\hdots,x_{n}) \ \
(a_{j} \in {\mathfrak m}^{2}, b_{j} \in {\mathfrak m})  \]
for any  $t \in {\mathbb C}^{*}$. By taking the limit when $t \to 0$ we
deduce that $X(x_{j})$ is of the form
$\alpha_{j}(x_{j+1},\hdots,x_{n})  + x_{j} \beta_{j}(x_{j+1},\hdots,x_{n})$
for some $\alpha_{j}  \in {\mathfrak m}^{2}$ and $\beta_{j} \in {\mathfrak m}$
and any $1 \leq j \leq n-1$.

The $2$-dimensional Lie algebra $U_{n}  + V_{n}$ is the Lie algebra of the
$2$-dimensional group
${\{ h^{\lambda,\mu} \}}_{(\lambda,\mu) \in {\mathbb C}^{*} \times {\mathbb C}}$.
The results of the two last paragraphs imply that ${\mathfrak g}$ is contained
in ${\mathcal G}_{0}$. Hence we get ${\mathfrak g} = {\mathcal G}_{0}$.

The equality
$l(H_{0}) = l({\mathfrak g}) = l({\mathcal G}_{0}) = 2n$ is a consequence of
prop.\ref{pro:lie}, \ref{pro:lieder} and cor. \ref{cor:res}.
Let $\alpha =(\alpha_{1}, \hdots, \alpha_{n}) \in H_{0}$. We define
\[ \alpha_{k} =(j^{k} \alpha_{1}, \hdots, j^{k} \alpha_{n-1}, \alpha_{n}) \]
for $k \geq 2$.
It is clear that $\alpha_{k}$ belongs to $G_{0}$ for any
$k \geq 2$ and that $\alpha_{k}$ converges to $\alpha$ in the Krull topology when $k \to \infty$.
Thus $H_{0}$ is the closure $\overline{G}_{0}$ of $G_{0}$ in the Krull topology. We have
\[ \overline{G_{0}^{(2n)}} = \overline{H}_{0}^{(2n)} = \{Id \} \ {\rm and} \
 \overline{G_{0}^{(2n-1)}} = \overline{H}_{0}^{(2n-1)} \neq  \{Id \} . \]
Thus $G_{0}$ is a connected solvable group such that $l(G_{0})=2n$.
%
%
%
\end{proof}
\begin{cor}
\label{cor:exa}
Fix $n \in {\mathbb N}$. Then $G_{0}^{(1)}$ is a
unipotent solvable subgroup of $\diff{}{n}$ such that $l(G_{0}^{(1)})=2n-1$.
\end{cor}
\begin{proof}
The group $G_{0}^{(1)}$ is solvable and
\[ l(G_{0}^{(1)})= l (G_{0}) -1=2n-1. \]
Moreover it is unipotent, indeed it is composed of
tangent to the identity diffeomorphisms.
\end{proof}
\subsection{Nilpotent groups}
Now we give an example of an unipotent nilpotent subgroup $G$ of $\diffh{}{n}$ with
$l(G)=n$. In fact it suffices to find a nilpotent Lie algebra ${\mathfrak g}$
of ${\mathcal X}_{N} \cn{n}$ such that, $l({\mathfrak g})=n$. Thus, fix $n \in {\mathbb N}$
and take the following meromorphic functions:
\[
u_{n - 1} = x_n^{-1}, \hspace{0.2cm}u_{n - k} = x_{n - k + 1}^{-2}.u_{n - k + 1}^2,
\hspace{0.2cm} k= 2,\ldots, n-1.
\]
We consider the following vector fields in ${\mathcal X}_{N} \cn{n}$:
\begin{itemize}
\item
$X_1 = u_1^{-1}x_2^2\frac{\partial}{\partial x_1}$
\item
$X_2 = u_2^{-1}(4x_1x_2\frac{\partial}{\partial x_1} + x_2^2\frac{\partial}{\partial x_2})$
\item
$X_3 = u_3^{-1}(-4x_1x_3\frac{\partial}{\partial x_1} - 2x_2x_3\frac{\partial}{\partial x_2}
 + x_3^2\frac{\partial}{\partial x_3})$
\item
$X_k = u_k^{-1}(-2x_{k-1}x_k\frac{\partial}{\partial x_{k-1}} + x_k^2\frac{\partial}{\partial x_k})$,
\hspace{0.2cm} $k= 4,\ldots, n-1$.
\item
$X_n = -x_{n-1}x_n\frac{\partial}{\partial x_{n-1}} + x_n^2\frac{\partial}{\partial x_n}$
\end{itemize}

Now we see that they satisfy the following properties

\begin{lem}
\label{prop. exemplo}
Let $u_k$ and $X_k$ as above, then:
\begin{enumerate}
\item
$[X_k,X_j] = 0$, for $k,j\in\{1,\ldots,n\}$
\item
$X_n(u_{n - 1}) = -1$ and $X_j(u_{n - 1}) = 0$, for $j\in\{1,\ldots,n - 1\}$
\item
$X_j(u_{n - k}) = 0$, $j\neq n - k +1$,  $k\in\{2,\ldots,n - 1\}$
\item
$X_{n - k + 1}(u_{n - k}) = -2x_{n - k + 1}^{-1}.u_{n - k + 1}$ and $X_{n - k + 1}^2(u_{n - k}) = 2$
\item
$X_j(X_{n - k + 1}(u_{n - k})) = 0$, $j\neq n - k +1$, $k\in\{2,\ldots,n - 1\}$
\item
$(X_{n - k + 1}(u_{n - k}))^2 = 4u_{n - k}$.
\end{enumerate}
\end{lem}
\begin{proof}
The proof of all the properties besides (1)
is obtained by doing straightforward calculations.

Let us prove $[X_{k},X_{j}] \equiv 0$
for all $1 \leq j,k \leq n$. A formal vector field vanishing in
$x_{1}$, $u_{1}$, $\hdots$, $u_{n-1}$ vanishes
in $x_{1}$, $\hdots$, $x_{n}$ and then it is identically $0$.
Thus it suffices to prove that $[X_{k},X_{j}](x_{1}) = 0$
and $[X_{k},X_{j}](u_{l}) = 0$ for any $1 \leq l \leq n-1$.
The equality $[X_{k},X_{j}](x_{1}) = 0$ can be checked out
directly for all $1 \leq j,k \leq n$.
Property (2) implies  $[X_{k},X_{j}](u_{n-1}) = 0$
for all $1 \leq j,k \leq n$.
Moreover we have $[X_{k},X_{j}](u_{l}) = 0$
for all $1 \leq j,k \leq n$ and $1 \leq l < n-1$ by properties (3), (4) and (5).
\end{proof}

We denote
\[
Z_1 = u_1X_1, \ldots, Z_{n - 1} = u_{n - 1}X_{n - 1}, Z_{n}=X_n.
\]
We define ${\mathfrak g}$ the complex Lie algebra generated by
$\{Z_{1},\hdots,Z_{n}\}$.
We say that an element $Y$ of ${\mathfrak g}$ is a monomial of degree $1$ if
$Y \in \{Z_{1},\hdots,Z_{n}\}$. A monomial $Y_{k_{1},\hdots,k_{j}}$
of degree $j$ for $(k_{1},\hdots,k_{j}) \in {\{1,\hdots,n\}}^{j}$
is obtained by defining
\[ Y_{k_{1}}= Z_{k_{1}}, \ Y_{k_{1}, k_{2}} = [Z_{k_{2}}, Y_{k_{1}}],
\hdots, Y_{k_{1},\hdots,k_{j}} = [Z_{k_{j}}, Y_{k_{1},\hdots,k_{j-1}}] . \]
\begin{defi}
We say that a monomial $Y_{k_{1},\hdots,k_{j}}$
is a good monomial if
$k_{1} = \min (k_{1},\hdots,k_{j})$.
\end{defi}
The complex vector space generated by the monomials of degree $j$ is also generated by
the good monomials of degree $j$.
This is a consequence of the Jacobi's identity for  Lie brackets.
\begin{rem}
\label{rem:setup}
The Lie algebra ${\mathfrak g}$ is the complex
vector space generated by all good monomials. Moreover
${\mathcal C}^{j} {\mathfrak g}$ is the vector space generated by all good monomials of
degree greater or equal than $j+1$ for any $j \in {\mathbb N} \cup \{0\}$.
\end{rem}
\begin{pro}
\label{exemplo nilp}
The Lie algebra ${\mathfrak g}$ is nilpotent with $l({\mathfrak g})=n$ and
nilpotency class $3 \cdot  2^{n - 2} - 1$.
\end{pro}
The value of the nilpotency class of ${\mathfrak g}$ is not surprising.
We have ${\mathcal G}^{(j)} \subset {\mathcal C}^{2^{j}} {\mathcal G}$
for all Lie algebra ${\mathcal G}$ and $j \in {\mathbb N}$.
Thus a Lie algebra of soluble length $n$ is of nilpotent class at least
$2^{n-1}+1$. This sequence is of the same order than $3 \cdot  2^{n - 2} - 1$.

Denote
${\mathcal M}_{j}= \{ v \in \hat{K}_{n} : X_{1}(v)= \hdots = X_{j}(v) = 0 \}$
for $j$ in  $\{1,\hdots,n-1\}$.
\begin{lem}
\label{lem:fi}
Let $1 \leq j, k \leq n$. We have
$Z_{k}({\mathcal M}_{j}) \subset  {\mathcal M}_{j}$.  In particular
any element of ${\mathfrak g}$ of the form $v X_{k}$ with
$v \in \hat{K}_{n}$ satisfies $v \in {\mathcal M}_{k}$.
\end{lem}
\begin{proof}
We have
\[ (X_{k} \circ X_{l}) (v) = (X_{l} \circ X_{k} +[X_{k},X_{l}])(v) = X_{l} (X_{k}(v)) \]
for all $v \in \hat{K}_{n}$, $1 \leq k,l \leq n$.
We obtain $X_{k}({\mathcal M}_{j}) \subset  {\mathcal M}_{j}$
for all $1 \leq j,k \leq n$. Since $u_{k} \in {\mathcal M}_{j}$ for any $k \geq j$ then
$Z_{k}({\mathcal M}_{j}) \subset  {\mathcal M}_{j}$
for all $1 \leq j,k \leq n$.

A good monomial  $Y_{k_{1},\hdots,k_{j}}$ satisfies
\begin{equation}
\label{equ:red}
 Y_{k_{1},\hdots,k_{p}} =
 [Z_{k_{p}}, Y_{k_{1},\hdots,k_{p-1}}]= (Z_{k_{p}} \circ \hdots \circ Z_{k_{2}}) (u_{k_{1}}) X_{k_{1}}
 \end{equation}
for any $2 \leq p \leq j$.
Since $u_{k_{1}} \in {\mathcal M}_{k_{1}}$ then
$Y_{k_{1},\hdots,k_{j}}$ is of the form $v X_{k_{1}}$
for some $v \in {\mathcal M}_{k_{1}}$. The general result is
a consequence of remark \ref{rem:setup}.
%
\end{proof}
\begin{defi}
Let $v \neq 0$ be an element of $\hat{K}_{n}$.
We define $a(v)$ as the maximum of $j \in {\mathbb N} \cup \{0\}$ such that there exists
$(k_{1},\hdots,k_{j}) \in {\{1,\hdots,n\}}^{j}$ with
$(Z_{k_{j}} \circ \hdots \circ Z_{k_{1}}) (v) \neq 0$.
If the maximum does not exist we denote $a(v) = \infty$.
\end{defi}
\begin{lem}
\label{lem:length}
The nilpotency class of ${\mathfrak g}$ is
$\max (a(u_{1}), \hdots, a(u_{n-1})) + 1$.
\end{lem}
\begin{proof}
The equation
$(Z_{k_{l}} \circ \hdots \circ Z_{k_{1}}) (u_{j}) \neq 0$
implies $j < \min (k_{1}, \hdots,k_{l})$ by lemma \ref{lem:fi}.
Remark \ref{rem:setup} and equation (\ref{equ:red}) imply
that the nilpotency class of ${\mathfrak g}$ is   equal to
$\max (a(u_{1}), \hdots, a(u_{n-1})) + 1$.
\end{proof}
\begin{lem}
\label{lem:a}
The function $a$ satisfies
\begin{itemize}
\item
$a(uv)\leq a(u) + a(v) \ \forall u,v \in \hat{K}_{n}$.
\item
$a(u_{n - k}) = 3 \cdot 2^{k - 1} - 2 \ \forall 1 \leq k \leq n-1$.
\end{itemize}
\end{lem}
\begin{proof}
The first result is a consequence of Leibnitz's formula.
More precisely $Z_{r_1} \circ \ldots \circ Z_{r_l}(uv)$ is a sum of elements of the form
\begin{equation}
\label{equ:leibnitz}
(Z_{r_{s_1}} \circ
\ldots \circ Z_{r_{s_j}}(u))(Z_{r_{s_{j + 1}}} \circ \ldots \circ Z_{r_{s_l}}(v))
\end{equation}
where
$s_{1} < \hdots < s_{j}$, $s_{j+1} < \hdots < s_{l}$ and
$\{s_1,\ldots,s_l\} = \{1,\ldots,l\}$.
If $l = a(u) + a(v) +1$ then either $j \geq a(u) +1$ or $l-j \geq a(v) +1$ and
all the terms of the form (\ref{equ:leibnitz}) are vanishing.
We deduce $a(uv) \leq a(u) + a(v)$.

We have $Z_{n}(u_{n-1})= -1$ and $Z_{j}(u_{n-1})=0$ for any $1 \leq j \leq n-1$
by property (2) in lemma \ref{prop. exemplo}.
We obtain $a(u_{n-1})=1=3 \cdot 2^{1- 1} - 2 $.
Let us prove the second result by induction on $k$.
Suppose the result is true for $k-1$.
We have $Z_{j}(u_{n-k})=0$ for any $j \neq n-k+1$ by property (3) in
 lemma \ref{prop. exemplo}. Since
$Z_{n - k + 1}(u_{n - k}) = u_{n - k + 1}X_{n - k + 1}(u_{n - k})$ we obtain
 \[
 a(u_{n - k}) =
 1 + a(Z_{n - k + 1}(u_{n - k})  ) \leq 1 + a(u_{n - k + 1}) + a(X_{n - k + 1}(u_{n - k})).
 \]
Now again
 $Z_{j}(X_{n - k + 1}(u_{n - k}))\neq 0$ implies $j = n-k+1$ by lemma \ref{prop. exemplo}.
 Since  $Z_{n - k + 1}(X_{n - k + 1}(u_{n - k})) = 2u_{n - k + 1}$ we obtain
 \[ a(X_{n - k + 1}(u_{n - k})) = 1 + a(2u_{n - k + 1}) \implies
 a(u_{n - k})  \leq 2 + 2 a(u_{n - k + 1})   . \]
 This leads us to
 $a(u_{n - k})\leq 2 + 2a(u_{n - (k - 1)}) = 3 \cdot 2^{k - 1} - 2$
 by induction hypothesis.

To prove the equality it suffices to show that
\begin{equation}
\label{equ:chain}
(Z_{n}^{2^{k - 1}} \circ  Z_{n - 1}^{2^{k - 1}} \circ Z_{n - 2}^{2^{ k - 2}} \circ \ldots \circ
Z_{n - k + 3}^{2^3} \circ Z_{n - k + 2}^{2^2} \circ Z_{n - k + 1}^2)(u_{n - k}) \in \mathbb{C}^{*}
\end{equation}
since it is a composition of  $3 \cdot 2^{k - 1} - 2$ derivations applied to $u_{n-k}$.
We claim that
$X_{n}^{m}(u_{n - 1}^m)\in\mathbb{C}^*$ and
$X_{n - k + 1}^{2m}(u_{n - k}^m)\in\mathbb{C}^*$, $k= 2,\ldots, n - 1$ and $m \in {\mathbb N}$.
The first property is a consequence of property (2) in lemma \ref{prop. exemplo}.
The second property is deduced of condition (4) in  lemma \ref{prop. exemplo}
for $m=1$. The equation
\[ X_{n - k + 1}^{2(m + 1)}(u_{n - k}^{m + 1}) = (X_{n - k + 1}^{2m } \circ X_{n - k + 1})
((m+1) u_{n - k}^{m} X_{n-k+1}(u_{n - k})) = \]
\[ X_{n - k + 1}^{2m}
(  (m+1) [m u_{n - k}^{m-1} X_{n-k+1}(u_{n - k})^{2} +
u_{n - k}^{m} X_{n-k+1}^{2}(u_{n - k})] )= \]
\[ X_{n - k + 1}^{2m}
(  (m+1) (4m+2) u_{n-k}^{m} )  \]
implies that it holds true for any $m \in {\mathbb N}$ by induction.
We obtain
$Z_{n}^{m}(u_{n - 1}^m)\in\mathbb{C}^*$ and
$Z_{n - k + 1}^{2m}(u_{n - k}^m)\in\mathbb{C}^{*} u_{n-k+1}^{2m}$ for
$k= 2,\ldots, n - 1$ and $m \in {\mathbb N}$.
The proof of equation (\ref{equ:chain}) is now straightforward.
\end{proof}
\begin{proof}[proof of prop. \ref{exemplo nilp}]
The Lie algebra ${\mathfrak g}$ is of nilpotent class
$3 \cdot 2^{n - 2} - 1$ by lemmas \ref{lem:length} and \ref{lem:a}.
We obtain
\[  {\mathfrak g}^{(j)} \subset {\mathcal M}_{1} X_{1} + \hdots +  {\mathcal M}_{n-j}  X_{n-j} \
\forall 0 \leq j \leq n-1 \]
and ${\mathfrak g}^{(n)} = \{0\}$ by lemma \ref{lem:fi}.
Moreover  since
\[
\{X_{n - 1},X_{n - 1}(u_{n - 2})X_{n - 2},\ldots,
 X_{n - 1}(u_{n - 2})X_{n - 2}(u_{n - 3})\ldots X_{2}(u_1)X_1\} \subset {\mathfrak g}^{(1)}
\]
\[
\{X_{n - 2},X_{n - 2}(u_{n - 3})X_{n - 3},\ldots,
 X_{n - 2}(u_{n - 3})\ldots X_{2}(u_1)X_1\} \subset {\mathfrak g}^{(2)}
\]
\[
\vdots
\]
\[
\{X_{2},X_{2}(u_{1})X_{1}\} \subset {\mathfrak g}^{(n - 2)}
\]
\[
\{X_{1}\} \subset  {\mathfrak g}^{(n - 1)}.
\]
we obtain $l({\mathfrak g}) = n$.
\end{proof}
\begin{pro}
Fix $n \in {\mathbb N}$.
The set $G = {\rm exp}({\mathfrak g})$ is a unipotent nilpotent subgroup of
$\diff{}{n}$ such that $l(G)=n$.
\end{pro}
\begin{proof}
Since ${\mathfrak g}$ is nilpotent it is also a finite dimensional complex vector space.
Thus ${\mathfrak g}$ is closed in the Krull topology.

We define ${\mathfrak g}_{k} = \{ X_{k} : X \in {\mathfrak g} \}$ (see equation
(\ref{equ:actvf})) for $k \in {\mathbb N}$. It is a nilpotent Lie algebra of nilpotent elements.
Such a Lie algebra is always algebraic (pag. 105-108 and specially corollary 7.7
\cite{Borel}), i.e. ${\mathfrak g}_{k}$ is the Lie algebra of a linear algebraic group
$A_{k} \subset GL({\mathfrak m}/{\mathfrak m}^{k+1})$ for any $k \in {\mathbb N}$.
Up to a change of coordinates ${\mathfrak g}_{k}$ is a Lie algebra of upper
triangular matrices with vanishing principal diagonal (th. \ref{teo:Lie}).
Thus $A_{k}$ is composed of unipotent elements and
${\rm exp}: {\mathfrak g}_{k} \to A_{k}$ is a bijection for any $k \in {\mathbb N}$.
We deduce that $G =  {\rm exp}({\mathfrak g})$ is a unipotent subgroup of
$\diff{}{n}$. Indeed since ${\mathfrak g}_{1}= 0$ we have $j^{1} G = \{Id\}$.

Consider the notations in the paragraph preceeding definition \ref{def:liegrp}.
It is clear that $C_{k}=A_{k}$ and then $G_{k}=A_{k}$ since $A_{k}$ is algebraic
for any $k \in {\mathbb N}$.  Since ${\mathfrak g}$ is closed in the Krull topology then
so is $G$.  Thus we obtain that $G = \overline{G}^{(0)}$ and
${\mathfrak g}$ is the Lie algebra of the unipotent connected group $\overline{G}^{(0)}$.
This implies
$l(G) = l(\overline{G}^{(0)})=l({\mathfrak g}) = n$ by prop.
\ref{pro:lieder} and \ref{exemplo nilp}.
\end{proof}
\begin{rem}
The examples can be chosen to be finitely generated. Given a group
$G$ of finite length $p$ there exists a finitely generated subgroup
$\tilde{G}$ of $G$ such that $l(\tilde{G})=p$. Indeed
there exist $1 \neq \phi \in G^{(p)}$ and a finitely generated
group $\tilde{G} = <\phi_{1}, \hdots, \phi_{q} >$ such that
$\phi \in \tilde{G}^{(p)}$. We obtain
$p \leq l(\tilde{G}) \leq l(G) =p$.
\end{rem}
\bibliography{rendu}
\end{document}